\documentclass[12pt, a4paper]{amsart}

\usepackage[hmargin=30mm, vmargin=25mm, includefoot, twoside]{geometry}
\usepackage[bookmarksopen=true]{hyperref}

\usepackage{amsfonts,amssymb,verbatim}
\usepackage{latexsym}
\usepackage{mathrsfs}
\usepackage{stmaryrd}
\usepackage{xspace}
\usepackage{enumerate}
\usepackage{paralist}
\usepackage{graphicx}
\usepackage[all]{xy}
\usepackage{extarrows}
\usepackage[usenames,dvipsnames]{xcolor}

\usepackage{txfonts, pxfonts}

\usepackage{amsmath}
\usepackage{amsthm}

\usepackage[T1]{fontenc}	
\usepackage[utf8]{inputenc}			
\usepackage{mathtools} 					
\usepackage{microtype} 					
\usepackage{bbm}					
\usepackage{bm}						
\usepackage[shortcuts]{extdash} 			
\usepackage{subcaption}					
\usepackage{mparhack}					
\usepackage{accents}

\usepackage{todonotes}
\newcommand{\CCC}{\mathbb{C}}

		\newcommand{\NN}{\mathbb{N}}	
			
		\newcommand{\RR}{\mathbb{R}}

\newcommand{\CC}{\mathcal{C}}			
		\newcommand{\CF}{\mathcal{F}}	
\newcommand{\CG}{\mathcal{G}}		\newcommand{\CH}{\mathcal{H}}

		\newcommand{\CP}{\mathcal{P}}

\newcommand*{\note}[1]{\todo[backgroundcolor=white]{\footnotesize#1}}

\newcommand{\ubar}[1]{\underaccent{\bar}{#1}} 		

\newcommand{\cay}{{\rm Cay}}

\DeclarePairedDelimiter\abs{\lvert}{\rvert}		
\DeclarePairedDelimiter\norm{\lVert}{\rVert}		
\DeclarePairedDelimiter\angles{\langle}{\rangle}	
\DeclarePairedDelimiter\paren{(}{)}			
\DeclarePairedDelimiter\brackets{[}{]}			
\DeclarePairedDelimiter\braces{\{}{\}}			

	\newcommand{\Bigparen}[1]{\paren[\Big]{#1}}

\newcommand{\Bigbrack}[1]{\brackets[\Big]{#1}}
	\newcommand{\bigbraces}[1]{\braces[\big]{#1}}
	\newcommand{\Bigbraces}[1]{\braces[\Big]{#1}}
\newcommand{\bigmid}{\mathrel{\big|}}			
\newcommand{\Bigmid}{\mathrel{\Big|}}			

\DeclareMathOperator{\id}{id}				

\DeclareMathOperator{\aut}{Aut}				

\theoremstyle{plain}
\newtheorem{thm}{Theorem}[section]				
\newtheorem{prop}[thm]{Proposition}		\newtheorem{proposition}[thm]{Proposition}
\newtheorem{lem}[thm]{Lemma}						
\newtheorem{cor}[thm]{Corollary}		
\newtheorem{qu}[thm]{Question}

\newtheorem*{thm*}{Theorem}			\newtheorem*{theorem*}{Theorem}		
\newtheorem*{prop*}{Proposition}		\newtheorem*{proposition*}{Proposition}
\newtheorem*{lem*}{Lemma}			\newtheorem*{lemma*}{Lemma}			
\newtheorem*{cor*}{Corollary}			\newtheorem*{corollary*}{Corollary}
\newtheorem*{qu*}{Question}			\newtheorem*{question*}{Question}
\newtheorem*{conj*}{Conjecture}			\newtheorem*{conjecture*}{Question}
\newtheorem*{fact*}{Fact}
\newtheorem*{claim*}{Claim}

\newtheorem{alphthm}{Theorem}			

\newtheorem{alphcor}[alphthm]{Corollary}

\theoremstyle{definition}
\newtheorem{de}[thm]{Definition}

\newtheorem*{de*}{Definition}			\newtheorem{definition*}{Definition}	
\newtheorem*{notation*}{Notation}	
\newtheorem*{conv*}{Convention}			\newtheorem*{convention*}{Convention}

\theoremstyle{remark}
\newtheorem{rmk}[thm]{Remark}			\newtheorem{remark}[thm]{Remark}			
			\newtheorem{example}[thm]{Example}

\numberwithin{equation}{section}

 \theoremstyle{definition}
  \newtheorem{defn}[thm]{Definition}
 \theoremstyle{remark}
 \newtheorem{rem}[thm]{Remark}

\def\B{\mathfrak B}
\def\K{\mathfrak K}

\def\H{\mathcal H}

\def\supp{\mathrm{supp}}
\def\diam{\mathrm{diam}}

\def\Ad{\mathrm{Ad}}

\def\N{\mathbb N}
\def\C{\mathbb C}

\def\Cq{C^*_{uq}}

\newcommand{\fol}{F\o{}lner\xspace}
\newcommand{\fnc}{\ubar c}

\begin{document}

\title{On the structure of asymptotic expanders}

\author{Ana Khukhro, Kang Li, Federico Vigolo, and Jiawen Zhang}

\address[A. Khukhro]{Department of Pure Mathematics and Mathematical Statistics, University of Cambridge, Wilberforce Road, Cambridge CB3 0WB, United Kingdom.}
\email{a.khukhro@dpmms.cam.ac.uk}

\address[K. Li]{Institute of Mathematics of the Polish Academy of Sciences, \'{S}niadeckich 8, 00-656 Warsaw, Poland.}
\email{kli@impan.pl}

\address[F. Vigolo]{Weizmann Institute of Science, 234 Herzl Street, 7610001, Rehovot, Israel.}
\email{vigolo.maths@gmail.com}

\address[J. Zhang]{School of Mathematics, University of Southampton, Highfield, SO17 1BJ, United Kingdom.}
\email{jiawen.zhang@soton.ac.uk}

\date{}
\subjclass[2010]{Primary: 46H35, 20F65, 05C99, 51F99. Secondary: 19K56.}
\keywords{Asymptotic Expanders; \fol Sets; Averaging Projection; Coarse Baum--Connes Conjecture; Uniform Roe Algebra; Quasi-Locality.}

\thanks{This work was partially supported by the grant 346300 for IMPAN from the Simons Foundation and the matching 2015-2019 Polish MNiSW fund. KL was supported by the European Research Council (ERC) under the European Union's Horizon 2020 research and innovation programme (grant agreement no. 677120-INDEX). KL was also supported by the Internal KU Leuven BOF project C14/19/088. FV was supported by the ISF Moked 713510 grant number 2919/19. JZ was supported by the Sino-British Trust Fellowship by Royal Society, International Exchanges 2017 Cost Share (China) grant EC$\backslash$NSFC$\backslash$170341, and NSFC11871342.}

\begin{abstract}
In this paper, we study several analytic and graph-theoretic properties of asymptotic expanders. We show that asymptotic expanders can be characterised in terms of their uniform Roe algebra. Moreover, we use asymptotic expanders to provide uncountably many new counterexamples to the coarse Baum--Connes conjecture. 
Finally, we show that vertex-transitive asymptotic expanders are actually expanders. In particular, this gives a $C^*$-algebraic characterisation of expanders for vertex-transitive graphs. We achieve these results by showing that a sequence of asymptotic expanders always admits a ``uniform exhaustion by expanders''. This also implies that asymptotic expanders cannot be coarsely embedded into any $L^p$-space.
\end{abstract}

\date{\today}
\maketitle

\parskip 4pt

\section{Introduction}
This paper is concerned with the structure theory of asymptotic
expanders. The original motivation for this work comes from the study
of uniform Roe algebras associated with metric spaces: we developed a
structure theory for asymptotic expanders in order to answer some
questions raised in \cite{intro}. While doing so, we unveiled
structural properties of asymptotic expanders which appear to be of
independent interest and lead to more counterexamples to the coarse
Baum--Connes conjecture. The definition of asymptotic expanders is
elementary and of attractively geometric nature. In turn, many of the
techniques employed in this paper have a
graph-theoretic/coarse-geometric flavour. 

\subsection{Original $C^*$\=/algebraic motivation}
Asymptotic expanders were introduced in \cite{intro}. The driving motivation behind this definition was the study of the relation between the \emph{uniform Roe algebra} $C^*_u(X)$ and the \emph{uniform quasi\=/local algebra} $C^*_{uq}(X)$ for a discrete metric space $X$ (see Subsection~\ref{ssec:Roe.algebras} for precise definitions of these algebras). In short, the uniform Roe algebra of a discrete metric space $X$ with \emph{bounded geometry} is a $C^*$\=/algebra that carries a great deal of coarse geometric information about the underlying metric space $X$, and is a key component in a number of significant conjectures relating coarse geometric/topological properties with analytic ones (recall that a discrete metric space $X$ has bounded geometry if there is, for every $R>0$, a uniform upper bound on the cardinalities of all the $R$-balls in $X$).

Although the uniform Roe algebra $C^*_u(X)$ plays a key role in many conjectures, it is usually difficult to determine whether a given operator belongs to $C^*_u(X)$.
In an attempt to find a more practical characterisation of elements in $C^*_u(X)$, Roe introduced the notion of quasi-locality in \cite{Roe88, Roe96}. The quasi\=/local operators form an algebra---denoted $C^*_{uq}(X)$---that is somewhat easier to handle, and it follows from the definition that $C^*_u(X)\subseteq C^*_{uq}(X)$.
In some special cases, it has been proved that these two algebras coincide \cite{Eng19,ST19,SZ18}, but it is not known whether this is always the case (we refer to \cite{intro, SZ18} and references therein for further discussion and motivation).

The idea in \cite{intro} was to focus on the study of the \emph{averaging projection} $P_X$, where $X$ is a \emph{coarse disjoint union} of a sequence of finite metric spaces (see Subsections~\ref{ssec:basic.notation} and \ref{ssec:averaging.projections}). One of the main results in \cite{intro} is a complete characterisation of the metric spaces for which the averaging projection $P_X$ is quasi\=/local. Namely, if $X=\bigsqcup_{n\in\NN} X_n$ is the coarse disjoint union of a sequence of finite metric spaces $\{X_n\}$ with $\abs{X_n}\to \infty$, then $P_X$ is quasi\=/local if and only if the following condition is satisfied:

\noindent\parbox{3.8 em}{(AsEx)}\parbox{\textwidth-3.8 em}{for every $0<\alpha\leq\frac{1}{2}$, there exist $0<c<1$ and $R>0$, such that
     \(
      \abs{\partial_RA}> c\abs{A}
     \)
     for every $A\subset X_n$ satisfying $\alpha\abs{X_n}\leq \abs{A}\leq \frac{1}{2}\abs{X_n}$;}

\noindent where $\partial_R A$ is the \emph{$R$\=/boundary} of $A$, \emph{i.e.,} the set of points in the complement of $A$ that are at most at distance $R$ from $A$ (see \cite[Theorem B]{intro}). This characterisation led the authors of \cite{intro} to introduce the concept of asymptotic expanders:

\begin{de*}[{\cite[Definition 3.9]{intro}}]
 A sequence $\braces{X_n}_{n\in\NN}$ of finite metric spaces with $\abs{X_n}\to \infty$ is a sequence of \emph{asymptotic expanders} if it satisfies the condition (AsEx).
\end{de*}

\begin{rmk}
 The notion of asymptotic expanders is a weakening of the notion of expanders (Definition~\ref{de:expanders}). This is of course no surprise, as it is well\=/known that the averaging projection $P_X$ belongs to the uniform Roe algebra $C^*_u(X)$ (hence is quasi\=/local) when $X$ is a coarse disjoint union of expanders (see for example \cite[Corollary~3.4]{intro})\footnote{Starting (at least) with the work of Voiculescu \cite{MR1026768}, it has been known that expander-like properties can be used to construct projections in $C^*_u(X)$ with pathological properties. }. 
\end{rmk}

The authors of \cite{intro} had excellent reasons to focus on averaging projections. In fact, if $\abs{X_n}\to\infty$ then the averaging projection $P_X$ is always a non\=/compact ghost operator (Definition~\ref{de:ghost.operator}). In particular, if $P_X$ is also quasi\=/local then we are in the following win-win situation:
\begin{itemize}
 \item either $P_X$ does not belong to the uniform Roe algebra---in which case we just proved that $C^*_u(X)\neq C^*_{uq}(X)$;
 \item or $P_X$ is in the uniform Roe algebra---in which case we are in a good position for providing further counterexamples to the coarse Baum--Connes conjecture.\footnote{A key ingredient in the construction of counterexamples to the coarse Baum--Connes conjecture is the existence of a non\=/compact ghost projection in the (uniform) Roe algebra.}
\end{itemize}

One possible way to distinguish whether $P_X\in C^*_u(X)$ or not is based on a result by Finn-Sell \cite{FinnSell2014}, from which one can deduce that $P_X\notin C^*_u(X)$ if $\braces{X_n}_{n\in\NN}$ \emph{uniformly coarsely embeds} into some Hilbert space (Subsection~\ref{ssec:basic.notation}).
This raises the question \cite[Question 7.3]{intro} of whether there exist asymptotic expanders with bounded geometry that can be coarsely embedded into Hilbert spaces.

Initially, the above was our motivating question. In this work, we provide a negative answer to \cite[Question 7.3]{intro} and various other more refined results.

\subsection{Further motivation for asymptotic expanders}
In the previous subsection we explained the origin of the notion of asymptotic expanders, but we came to realise---a posteriori---that this notion should be much more useful and natural than it appears at first sight.

For one, asymptotic expanders turn out to behave very much like expanders, hence share several nice properties. At the same time, the notion of asymptotic expanders is more ``robust''. For example, they are invariant under ``perturbations of linearly small scale''. This is emphatically not the case for expanders (see Section~\ref{sec:examples} for more details).

Secondly, the notion of asymptotic expanders arises naturally in the context of $C^*$\=/algebras as explained in the previous subsection, and in the context of dynamics since it captures exactly the notion of \emph{strong ergodicity} (in a sense made precise in \cite{li2021asymptotic}). The latter also provides more evidence to support the ``robustness'' of the notion of asymptotic expanders: it is well\=/known that strong ergodicity is invariant under fairly wild transformations (\emph{e.g.} orbit equivalences).

We deem that this mixture of naturality, robustness and expander-like behaviour will turn out to have unexpected applications in subjects ranging from dynamics of group actions to operator algebras and coarse geometry.

\subsection{The main structure results}
We prove two structure theorems for asymptotic expanders: one for finite graphs with uniformly bounded degree (Theorem~\ref{thm:structure.of.as.exp.}), and one in the more general setting of finite metric spaces (Theorem~\ref{thm:structure.of.as.exp.general.case}).

In the case of graphs we prove the following, showing that asymptotic expanders display expander-like behaviour in a certain precise sense (see Definition \ref{de:exhaustion} for the definition of ``uniform exhaustion by expanders'').

\begin{alphthm}[Theorem~\ref{thm:structure.of.as.exp.}]\label{thm:intro.structure.graph}
 Let $\{X_n\}_{n \in \N}$ be a sequence of finite connected graphs with uniformly bounded degree and $|X_n| \to \infty$.
Then it is a sequence of asymptotic expanders if and only if it admits a uniform exhaustion by expanders.
\end{alphthm}

The proof of the structure results (Theorem \ref{thm:structure.of.as.exp.} and Theorem \ref{thm:structure.of.as.exp.general.case}) is actually elementary. At its heart is the following key observation: if $F\subset X_n$ is a maximal (with respect to the inclusion) subset such that $\abs{\partial_RF}\leq c\abs{F}$ and $\abs{F}\leq \frac{1}{2}\abs{X_n}$, then any other $A\subset X_n$ that is disjoint from $F$ and has $\abs{A}\leq \frac{1}{2}\abs{X_n}-\abs{F}$ must have $\abs{\partial_R A\smallsetminus F}>c\abs{A}$ (Lemma~\ref{lem:complement.of.maximal.F}). The importance of this observation is that it implies the complement $X_n\smallsetminus F$ must satisfy a Cheeger inequality if $\abs{F}$ is very small.

\begin{rmk}
 Amine Marrakchi pointed out to us that these structure theorems---and their proofs---appear to be a graph-theoretic analogue of some previously known results for von Neumann algebras (see \emph{e.g.} \cite{marrakchi_strongly_2018}). We defer to \cite{li2020markovian,li2021asymptotic} for a more detailed discussion on this analogy. 
\end{rmk}

\subsection{Consequences of the Structure Theorems}
The Structure Theorems have a number of interesting applications.
First of all, we are able to provide a strongly negative answer to \cite[Question 7.3]{intro}:

\begin{alphthm}[Theorem \ref{thm: non CE general case}]\label{thm:intro:non.CE}
 A sequence of asymptotic expanders with uniformly bounded geometry cannot be coarsely embedded into any $L^p$\=/space for $1\leq p<\infty$ nor any \emph{uniformly curved} Banach space.
\end{alphthm}

\begin{rmk}
  Essentially, Theorem~\ref{thm:intro:non.CE} is a consequence of the fact that inside any sequence $\braces{X_n}_{n\in\NN}$ of asymptotic expanders with uniformly bounded geometry we can find \emph{weakly embedded} (see Section \ref{ssec:basic.notation}) sequences of actual expanders. Expanders are known to not coarsely embed into $L^p$-spaces or uniformly curved Banach spaces, and a standard argument shows that the existence of such a weak embedding is enough to prevent a coarse embedding into such Banach spaces (see Section~\ref{sec:non.coarse.embeddability} for a detailed explanation).
\end{rmk}

Theorem~\ref{thm:intro:non.CE} implies that it is impossible to use the strategy outlined in \cite[Section~7]{intro} to prove that the averaging projection $P_X$ is quasi-local but does not belong to the uniform Roe algebra of $X$.
This is no accident. In fact, we are then able to prove that if $\braces{X_n}_{n\in\NN}$ is a sequence of asymptotic expanders then the averaging projection $P_X$ belongs to the uniform Roe algebra $C^*_u(X)$. This closes the circle of ideas and proves the following:

\begin{alphthm}[Theorem \ref{thm: existence of ghost proj}]\label{thm:intro:averaging.operator.in.Roe}
Let $\{X_n\}_{n\in \N}$ be a sequence of finite metric spaces of uniformly bounded geometry with $|X_n| \to \infty$. Let $X$ be their coarse disjoint union, and $P_X$ be the averaging projection. Then the following are equivalent:
\begin{enumerate}
  \item  $P_X$ belongs to the uniform Roe algebra $C^*_u(X)$;
  \item $P_X$ belongs to the uniform quasi-local algebra $C^*_{uq}(X)$;
  \item $\{X_n\}_{n\in \N}$ is a sequence of asymptotic expanders.
\end{enumerate}
\end{alphthm}

As we remarked before, we are now in a very good position to produce more counterexamples to the coarse Baum--Connes conjecture. In fact, Finn-Sell proved that---under the mild condition of \emph{fibered coarse embedding}---metric spaces containing non\=/compact ghost projection in their Roe algebra do violate the coarse Baum--Connes conjecture \cite{FinnSell2014}. The averaging projection $P_X$ of a sequence of asymptotic expanders is a non\=/compact ghost projection and Theorem~\ref{thm:intro:averaging.operator.in.Roe} implies that it does belong to the uniform Roe algebra. This is already enough to produce some counterexamples; but we can actually get a more general result by showing that the existence of \emph{weakly embedded} asymptotic expanders (Definition~\ref{de:weak.embedding}) is a sufficient condition for the existence of non\=/compact ghost projections. Namely, we prove the following:

\begin{alphthm}[Theorem~\ref{thm: existence of ghost proj in Roe}]\label{thm:intro:counterexamples}
If a sequence of finite metric spaces with uniformly bounded geometry admits a \emph{fibered coarse embedding} into a Hilbert space and also contains a weakly embedded sequence of asymptotic expanders, then it violates the coarse Baum--Connes conjecture.
\end{alphthm}

Starting with a sequence of expanders of \emph{large girth}, we use a geometric trick to construct a sequence of asymptotic expanders that are \emph{not} expanders and have large girth (Example~\ref{eg:Schreier.triv.intersection}). This example satisfies the hypotheses of Theorem~\ref{thm:intro:counterexamples} (graphs with large girth do admit such a fibered coarse embedding) and hence provides new counterexamples to the coarse Baum--Connes conjecture. This provides an affirmative answer to Question 7.2 in \cite{intro}.

To the best of our knowledge, all the previously known counterexamples to the coarse Baum--Connes conjecture are based on some sort of spectral gap condition \cite{H99, HLS02, MR2871145, FinnSell2014, sawicki_warped_2017}. On the other hand, asymptotic expanders that are \emph{not} expanders do not satisfy this spectral gap condition (see Remark~\ref{rmk:counterexamples.no.gap} for a more in-depth discussion). Hence, we deduce our main corollary of this paper:

\begin{alphcor}[Corollary~\ref{cor:uncountably.many.counterexamples}]\label{cor:intro:main.corollary}
There are counterexamples to the coarse Baum--Connes conjecture that do not have spectral gaps. In fact, there exist uncountably many such examples that are not coarsely equivalent to one another.
\end{alphcor}

Asymptotic expanders that are not expanders cannot be ``highly symmetric''. More precisely, we prove the following:

\begin{alphthm}[Theorem~\ref{thm:symmetric.as.exp.iff.exp} and Remark~\ref{vertex-tran}]\label{thm:intro:symmetric.as.exp}
 If $\braces{X_n}_{n\in\NN}$ is a sequence of finite connected vertex-transitive graphs with uniformly bounded degree, then they are a sequence of asymptotic expanders if and only if they are a sequence of expanders.
\end{alphthm}

Note that the condition of being vertex\=/transitive is a rather generous hypothesis (Theorem~\ref{thm:symmetric.as.exp.iff.exp} has a more refined statement).

\begin{rmk}
 It is interesting to note that the typical counterexamples to the coarse Baum--Connes conjecture are constructed using box spaces. On the contrary, since expanders and asymptotic expanders are invariant under coarse equivalence \cite{intro}, Theorem~\ref{thm:intro:symmetric.as.exp} implies that the counterexamples constructed in Corollary~\ref{cor:intro:main.corollary} are not even coarsely equivalent to box spaces.
\end{rmk}

Combining Theorem~\ref{thm:intro:averaging.operator.in.Roe} and Theorem~\ref{thm:intro:symmetric.as.exp}, we are able to obtain a $C^*$-algebraic characterisation of expanders for finite vertex-transitive graphs:

\begin{alphcor}[Corollary~\ref{cor:symmetric.Roe}]
 A sequence of finite connected vertex-transitive graphs with uniformly bounded degree (\emph{e.g.} any \emph{box space}) is a family of expanders if and only if the averaging projection belongs to the uniform Roe algebra.
\end{alphcor}

\subsection*{Overview}
This paper has two fairly distinct souls: on the one side there is a combinatorial/geometric approach that is used to prove the Structure Theorems and provide examples, on the other side most of the applications have a strong operator-algebraic flavour. This peculiarity is reflected throughout the text, as various elements are explained to a level of detail that would be ill-suited to a paper aimed at specialists in the field.
The paper is structured as follows.

In Section \ref{sec:preliminaries}, we recall the basic notions needed to discuss expanders, asymptotic expanders, Roe algebras, quasi-local operators, and averaging projections.

In Section \ref{sec:structure}, we establish the fundamental structure theorems for asymptotic expanders, showing that they can be ``exhausted'' by a sequence of spaces displaying expander behaviour. This is done first in the case of graphs, and then more generally for finite metric spaces.

In Section \ref{sec:non.coarse.embeddability}, we give a comprehensive account of the lack of coarse embeddings into Banach spaces for expanders and asymptotic expanders.

Section \ref{sec:examples} is devoted to providing examples of asymptotic expanders that are not expanders.
We also discuss how to adapt these examples to realise them as Schreier coset graphs, possibly with large girth. We also briefly discuss the existence of uncountably many such examples up to coarse equivalence.

In Section \ref{sec:ghosts},
we show that the averaging operator $P_X$ belongs to the uniform Roe algebra if and only if $X$ is a sequence of asymptotic expanders (Theorem~\ref{thm:intro:averaging.operator.in.Roe}). Thanks to the work of Finn-Sell, this is also the main step in the proof of Theorem~\ref{thm:intro:counterexamples}.

In Section \ref{sec:homog}, we show that a graph satisfying a weak symmetry condition---namely, admitting more than one maximal subset realising the Cheeger constant---must have a sufficiently big subset realising the Cheeger constant. We use this to give a class of examples for which being an asymptotic expander is equivalent to being an expander.

In Section \ref{sec:questions}, we mention some questions arising from our work.

\subsection*{Acknowledgments}
We are very grateful to Masato Mimura for pointing out to us Gross' trick (Remark~\ref{rmk:gross.trick}) and for explaining several subtleties concerning $p$\=/Poincaré inequalities in Section \ref{sec:non.coarse.embeddability}. We would also like to thank Piotr Nowak and J\'{a}n \v{S}pakula for several illuminating discussions and comments. Finally, we thank Amine Marrakchi for bringing \cite{marrakchi_strongly_2018} to our attention and the anonymous referee for his/her helpful suggestions and comments.

\section{Preliminaries}\label{sec:preliminaries}

This paper is aimed at a rather diverse audience. For this reason, and to fix the notation, we will now recall some very basic notions.

\subsection{Basic Notation}\label{ssec:basic.notation}
Let $(X,d)$ be a discrete metric space, $x\in X$ and $R>0$. Denote by $B(x,R)$ the closed ball in $X$ with centre $x$ and radius $R$. For any $A \subseteq X$, denote by $|A|$ the cardinality of $A$, and $\partial_R A=\{x\in X\backslash A\mid d(x,A)\leq R\}$ the \emph{(outer) $R$-boundary of $A$}. Recall that a discrete metric space $(X,d)$ has \emph{bounded geometry} if $\sup_{x\in X}|B(x,R)|$ is finite for each $R>0$. For $A \subseteq X$, we use $\chi_A$ for the \emph{characteristic function} of $A$ and the symbol $\delta_x$ for $\chi_{\{x\}}$ for $x\in X$.

Let $Y \subseteq X$ be a subset equipped with the induced metric. In order to avoid confusion, given a subset $A\subseteq Y\subseteq X$, we denote by $\partial^X_R A$ the $R$-boundary of $A$ in $X$, and by $\partial^Y_R A$ the $R$-boundary of $A$ in $Y$. Note that $\partial^Y_R A=\partial^X_R A\cap Y$ and that $\partial^X_R A=\partial^Y_R A\sqcup(\partial^X_R A\smallsetminus Y)$.

When $X$ is an unoriented simplicial graph---\emph{i.e.,} without loops and multiple edges, we always equip $X$ with its path\=/metric. Note that to keep the notation simple, we identify graphs with their sets of vertices. For any vertex $x\in X$, we define its \emph{degree} to be the number of points adjacent to $x$. A graph is said to have \emph{degree bounded by $D$} for some $D>0$ if any vertex has degree bounded by $D$. It is clear that a graph has bounded degree \emph{if and only if} it has bounded geometry. For $A\subseteq X$, we write $\partial A$ for the $1$\=/boundary $\partial_1 A$ of $A$. A subset $Y\subset X$ can be naturally seen as a (possibly disconnected) subgraph of $X$ by keeping all the edges having both endpoints in $Y$ (that is, we identify $Y$ with the full subgraph spanned by it).

\begin{rmk}
 In the context of graphs there is a small subtlety. A subgraph $Y\subset X$ comes with a natural path metric, which need not coincide with the induced metric from $X$. In particular, $\partial_R^YA$ may change depending on the metric under consideration. Fortunately, this is not the case when $R=1$ which we focus on, because $Y$ is a full subgraph.
\end{rmk}

A sequence of graphs $\braces{X_n}_{n\in\NN}$ is said to have \emph{large girth} if the length of the shortest closed loop in $X_n$ tends to infinity as $n\to\infty$. If $F_d$ is the (non-abelian) free group on $d$ generators and $\Gamma <F_d$ is a subgroup, then the \emph{Schreier coset graph} is the graph having left cosets as vertices and edges of the form $(g\Gamma ,sg\Gamma )$ for $s$ in the generating set of $F_d$. The action of $F_d$ by right multiplication descends to a (vertex-transitive) action on the Schreier graph if and only if $\Gamma $ is normal in $F_d$. If $\braces{\Gamma _n}_{n\in\NN}$ is a sequence of finite index subgroups of $F_d$ so that the sequence of Schreier graphs has large girth then $\bigcap_{n\in\NN}\Gamma _n=\emptyset$. If $\Gamma _{n+1}\subset \Gamma _n$ for every $n$, then the converse holds as well.

In this paper, we mainly focus on a sequence of finite metric spaces $\{(X_n,d_n)\}_{n\in \N}$ with $|X_n| \to \infty$, and their \emph{coarse disjoint union} $(X,d)$. In other words, $X=\bigsqcup_{n \in \N} X_n$ as a set, the restriction of the metric $d$ on each $X_n$ coincides $d_n$ and it further satisfies:
\[
d(X_n, X_m) \geq n+m+\diam(X_n)+\diam(X_m).
\]
Note that the coarse disjoint union $\bigsqcup_{n \in \N} X_n$ has bounded geometry if and only if the metric spaces have \emph{uniformly} bounded geometry.

A function $f\colon X\to Y$ between metric spaces is a \emph{coarse embedding} if there are two unbounded increasing \emph{control functions} $\rho_-,\rho_+\colon[0,\infty)\to[0,\infty)$ such that $\rho_-(d_X(x,x'))\leq d_Y(f(x),f(x'))\leq \rho_+(d_X(x,x'))$ for every $x,x'\in X$.
A coarse embedding $f\colon X\to Y$ is a \emph{coarse equivalence} if there exists a coarse embedding $g\colon Y\to X$ such that $f\circ g$ and $g\circ f$ are at bounded distance from the identity function.
If $E$ is a Banach space, $\braces{X_n}_{n\in\NN}$ a sequence of finite metric spaces and $X$ their coarse disjoint union, then $X$ admits a coarse embedding in $E$ if and only if the spaces $X_n$ \emph{uniformly} coarsely embed into $E$---that is, there are coarse embeddings $f_n\colon X\to E$ having the same control functions $\rho_-,\rho_+$ for every $n\in\NN$.

\begin{de}\label{de:weak.embedding} Following \cite{arzhantseva_relative_2015}, we say that a sequence of finite metric spaces $\{X_n\}_n$ admits a \emph{weak embedding} in a metric space $Y$ if there are $L$\=/Lipschitz functions $f_n\colon X_n\to Y$ such that for every $R>0$,
 \[
  \lim_{n\to\infty}\sup_{x\in X_n}\frac{\abs{f_n^{-1}(B(f_n(x),R))}}{\abs{X_n}}=0.
 \]
\end{de}

\subsection{\fol sets, expanders and asymptotic expanders}

\begin{de}
Let $X$ be a finite graph and $c>0$. A set $A\subset X$ is \emph{$c$\=/\fol} if $\abs{A}\leq\frac{1}{2}\abs{X}$ and $\abs{\partial A}\leq c\abs{A}$. A graph $X$ is a \emph{$c$\=/expander} if it has no non\=/empty $c$\=/\fol sets. The \emph{Cheeger constant} of $X$ is
 \[
  h(X)=\min\Bigbraces{\frac{\abs{\partial A}}{\abs{A}}\Bigmid A\subset X,\ 0<\abs{A}\leq\frac{1}{2}\abs{X}}
  =\min\bigbraces{c\bigmid \text{$\exists$ non\=/empty $c$\=/\fol set}}.
 \]
 Equivalently, $h(X)=\sup\braces{c\mid \text{$X$ is a $c$\=/expander}}$\footnote{To be precise, if $X$ is a disconnected graph then it is not a $c$\=/expander for any $c$, and we are hence taking the supremum of the empty set. In this case, we declare such a supremum to be $0$.}.
\end{de}

\begin{de}\label{de:expanders}
 A sequence of graphs $\braces{X_n}_{n\in\NN}$ with uniformly bounded degrees and increasing cardinalities is a \emph{sequence of expander graphs} if there exists a $c>0$ such that they are all $c$\=/expanders. That is, there is a $c>0$ small enough so that every $X_n$ does not have non\=/empty $c$\=/\fol sets (this condition can be rephrased as ``$X_n$ satisfies a linear isoperimetric inequality with constant $c$'').
\end{de}

\begin{rmk}
 In order to avoid trivialities, one would usually insist that \fol sets are non\=/empty. We decided not to do so because in this way every finite graph has a (possibly empty) maximal $c$\=/\fol set.
\end{rmk}

Recall the definition of asymptotic expanders (condition (AsEx) in the introduction); this notion admits two equivalent formulations, as follows.

\begin{de}[{\cite[Definition 3.9]{intro}}]\label{def: asymptotic expanders}
A sequence of finite metric spaces $\{X_n\}_{n \in \N}$ with $|X_n| \to \infty$ is said to be a sequence of \emph{asymptotic expanders} (or \emph{asymptotic expander graphs} when the $X_n$'s are connected graphs) if the following equivalent conditions (see \cite[Theorem 3.11]{intro}) hold:

\noindent\parbox{3.8 em}{(AsEx)}\parbox{\textwidth-3.8 em}{
    for every $\alpha \in (0,\frac{1}{2}]$, there exist $c_\alpha\in (0,1)$ and $R_\alpha>0$ such that for any $n \in \N$ and $A \subseteq X_n$ with $\alpha|X_n| \leq |A| \leq |X_n|/2$, we have $|\partial_{R_\alpha} A| > c_\alpha|A|$.}

 \noindent\parbox{3.8 em}{(AsEx')}\parbox{\textwidth-3.8 em}{
there exists $c\in (0,1)$, such that for any $\alpha \in (0,\frac{1}{2}]$, there exists $R_\alpha>0$ such that for any $n \in \N$ and $A \subseteq X_n$ with $\alpha|X_n| \leq |A| \leq |X_n|/2$, we have $|\partial_{R_\alpha} A| > c|A|$.}
\end{de}

\begin{rmk}
 A sequence of expander graphs is obviously a sequence of asymptotic expanders (graphs are seen as metric spaces when equipped with the path metric). As a matter of fact, the notion of asymptotic expanders is a substantial relaxation of that of expander in at least three respects. Firstly, the metric spaces under consideration need not be graphs (they need not even be geodesic metric spaces). Secondly, the $X_n$ need not have uniformly bounded geometry (so that, even if the $X_n$ were finite graphs, they could have increasing degrees). Finally, and perhaps most obviously, we are not requiring that the metric spaces satisfy a Cheeger inequality on the nose, but we allow some room for linearly-small subsets.
\end{rmk}

In the case of graphs with uniformly bounded degrees, the notion of asymptotic expanders is equivalent to the following.

\begin{lem}\label{lem: equiv def for asymptotic expanders(graph)}
Let $\{X_n\}_{n\in \NN}$ be a sequence of connected graphs with degrees uniformly bounded by $D$ and $\abs{X_n}\to\infty$. Then it is a sequence of asymptotic expander graphs \emph{if and only if} the following condition holds:

\noindent\parbox{3.8 em}{(AsEx'')}\parbox{\textwidth-3.8 em}{
for every $\alpha \in (0,\frac{1}{2}]$, there exists $c_\alpha\in (0,1)$ such that for every $n\in\NN$ and every subset $A\subset X_n$ with $\alpha\abs{X_n}\leq\abs{A}\leq\frac{1}{2}\abs{X_n}$, we have $\abs{\partial A}>c_\alpha\abs{A}$.}
\end{lem}

\begin{proof}
We only need to show that (AsEx) implies (AsEx''). By Definition \ref{def: asymptotic expanders}, for any $\alpha \in (0,\frac{1}{2}]$, there exist $c'_\alpha\in (0,1)$ and $R_\alpha>0$ such that for any $n \in \N$ and $A \subseteq X_n$ with $\alpha|X_n| \leq |A| \leq |X_n|/2$, we have $|\partial_{R_\alpha} A| > c'_\alpha|A|$. Note that for every subset $A\subset X$ we have $\abs{\partial(A\cup \partial A)}\leq D\abs{\partial A}$ because every vertex adjacent to $\partial(A\cup \partial A)$ must in fact be adjacent to $\partial A$. Therefore for any $R\in \NN$, we have
  \[
   \abs{\partial_R A}\leq (1+D+D^2+\cdots +D^{R-1})\abs{\partial A}< D^R \abs{\partial A}.
  \]
Taking $c_\alpha\coloneqq c'_\alpha D^{-R_\alpha}$, we finish the proof.
\end{proof}

\subsection{Roe algebras and quasi-local algebras}\label{ssec:Roe.algebras}

Let $X$ be a discrete metric space with bounded geometry, and $\H$ be a Hilbert space. An operator $T\in \B(\ell^2 (X;\H))$ can be viewed as an $X$-by-$X$ matrix $[T_{x,y}]_{x,y \in X}$ with $T_{x,y}\in \B(\H)$. Define the \emph{support} of $T$ to be $\supp(T):=\{(x,y)\in X \times X: T_{x,y} \neq 0\}$. We say that $T\in \B(\ell^2 (X;\H))$ has \emph{finite propagation} if there exists some constant $R >0$ such that $d(x,y) \leq R$ for every $(x,y) \in \supp(T)$. The smallest number $R$ satisfying this condition is called the \emph{propagation of $T$}. We say that $T$ is \emph{locally compact} if $T_{x,y}$ is a compact operator on $\H$ for every $x,y \in X$. Similarly, given discrete metric spaces $X,Y$, Hilbert spaces $\H, \H'$, and a bounded operator $V\colon \ell^2(X;\H) \to \ell^2(Y;\H')$, we define its \emph{support} to be $\supp(V)\coloneqq\{(y,x) \in Y \times X\mid \delta_y V \delta_x \neq 0\}$.

There is a natural multiplication representation $\rho\colon \ell^\infty(X) \hookrightarrow \B(\ell^2 (X;\H))$ defined by point-wise multiplication, \emph{i.e.} $\rho(f)(\xi)(x)\coloneqq f(x)\xi(x)$ for any $f\in \ell^\infty(X)$ and $\xi \in \ell^2 (X;\H)$.
In matrix notation, $\rho(f)$ is given by $[f(x)\delta\paren{x,y}\id_{\H}]_{x,y}$ where $\id_{\H}\in \B(\H)$ is the identity operator and $\delta\paren{x,y}=1$ if $x=y$ and $0$ otherwise.
In what follows, we will identify $\ell^\infty$ with its image in $\B(\ell^2(X;\H))$ and simply write $f$ instead of $\rho(f)$.
In particular, given $T\in\B(\ell^2(X;\H))$, the operator $\chi_A T\chi_B$ is the composition $\rho(\chi_A)T\rho(\chi_B)$. In matrix form, the element in position $(x,y)$ of $\chi_A T\chi_B$ coincides with $T_{x,y}$ if $x\in A$ and $y\in B$, and it is the zero operator otherwise.

When we take $\H$ to be the complex numbers $\C$, it is clear that each operator is locally compact, and the set of all finite propagation operators in $\B(\ell^2 (X))$ is a $\ast$-algebra, called the \emph{algebraic uniform Roe algebra} of $X$ and denoted by $\C_u [X]$. Its operator-norm closure in $\B(\ell^2(X))$ is called the \emph{uniform Roe algebra} of $X$, denoted by $C^*_u (X)$.

Similarly, when we take $\H$ to be the infinite-dimensional separable Hilbert space, the set of all finite propagation \emph{locally compact} operators in $\B(\ell^2 (X; \H))$ (\emph{i.e} operators such that $T_{xy}$ is compact for every $x,y\in X$) is a $\ast$-algebra. Its operator-norm closure in $\B(\ell^2(X; \H))$ is called the \emph{Roe algebra} of $X$, denoted by $C^*(X)$.

\begin{rmk}\label{rmk:uniformRoe.embeds.Roe}
 It is clearly possible to isometrically embed $\B(\CCC)$ into the space of compact operators $\K(\CH)\subset\B(\CH)$ (for example, by tensoring with a rank\=/$1$ projection $P_0\colon\CH\to\C$).
 In turn, this gives us a (non-canonical) isometric embedding $\iota\colon C^*_u(X) \to C^*(X)$ of the uniform Roe algebra into the Roe algebra. In matrix form, one such embedding can be described as sending $T=[T_{x,y}]_{x,y \in X}$ to $\iota(T)\coloneqq [T_{x,y} \otimes P_0]_{x,y \in X}$, where $P_0$ is a fixed rank one projection.
\end{rmk}

\begin{defn}[\cite{Roe88, Roe96}]\label{def: quasi-locality}
Let $R,\varepsilon > 0$, and $\H$ be a Hilbert space. An operator $T \in \B(\ell^2(X; \H))$ is said to have $(R,\varepsilon)$-\emph{propagation} if for any $A, B\subseteq X$ such that $d(A,B) \geq R$, we have	
\[
\| \chi_A T\chi_B\|\leq \varepsilon.
\]
If for all $\varepsilon > 0$, there exists $R > 0$ such that $T$ has $(R,\varepsilon)$-\emph{propagation}, then we say that $T$ is \emph{quasi-local} (this definition makes sense also when $X$ has does not have bounded geometry).
\end{defn}

It is routine to check that the set of all locally compact quasi-local operators in $\B(\ell^2(X; \H))$ form a $C^*$-subalgebra. When $\H=\C$ this $C^*$\=/algebra is called the \emph{uniform quasi-local algebra} of $X$, denoted by $\Cq(X)$; when $\H$ is the infinite-dimensional separable Hilbert space, it is called \emph{quasi-local algebra} of $X$, denoted by $C^*_q(X)$.

Since finite propagation operators are quasi-local, the (uniform) Roe algebra is contained in the (uniform) quasi-local algebra.
It is an open question whether these containments are equalities (equality is known to hold in some special cases, \emph{e.g.} if $X$ has \emph{Yu's Property A} \cite{SZ18}).

\subsection{Averaging projections}\label{ssec:averaging.projections}

Given a finite subset $F\subset Y$, the \emph{averaging projection of $F$} is the operator $P_F\in\B(\ell^2(Y))$ obtained as the orthogonal projection onto the span of $\chi_F \in \ell^2(Y)$. In matrix form, it can be represented by:
\begin{equation*}
(P_F)_{x,y}=
\begin{cases}
  ~1/|F|, & \text{if }x,y \in F; \\
  ~0, & \mbox{otherwise}.
\end{cases}
\end{equation*}
If $X=\bigsqcup_{n\in\NN}X_n$, then $\ell^2(X)$ is the Hilbert sum $\bigoplus_{n\in\NN}\ell^2(X_n)$. We use the following:

\begin{defn}\label{defn: averaging projection}
If $(X,d)$ is a coarse disjoint union of a sequence of finite metric spaces $\{(X_n,d_n)\}_{n \in \N}$, we define the \emph{averaging projection of $X$} to be
$P_X\coloneqq \bigoplus_{n \in \N} P_{X_n}.$
\end{defn}

\begin{defn}[G. Yu]\label{de:ghost.operator}
Let $X$ be a discrete metric space, $\H$ a Hilbert space and $T \in \B(\ell^2(X;\H))$. We say $T$ is a \emph{ghost operator} if for any $\varepsilon,R>0$, there exists a bounded $F \subseteq X$ such that for any $(x,y) \notin F \times F$ we have $\|\chi_{B(x,R)}T\chi_{B(y,R)}\|<\varepsilon$.
\end{defn}

\begin{remark}\label{rmk:average.non.compact.ghost}
It is clear that the projection $P_X$ in Definition~\ref{defn: averaging projection} is a non-compact ghost operator in $\B(\ell^2(X))$ provided that $\abs{X_n}\to\infty$.
\end{remark}

It is well\=/known that if $X$ is a sequence of expanders then $P_X$ is in the uniform Roe algebra $C^*_u(X)$ (see e.g. \cite[Examples~5.3]{MR2871145} or \cite[Corollary~3.4]{intro}). In particular, the (uniform) Roe algebra of a sequence of expanders contains a non-compact ghost projection---this fact was the key to producing counterexamples to the coarse Baum--Connes conjecture.

\begin{rmk}
 It is beyond the scope of this paper to enter into the statement of the coarse Baum--Connes conjecture. For us, it will be enough to say that it is a coarse analogue of the classical Baum--Connes conjecture and that it predicts that some assembly maps are isomorphisms. This conjecture has been proven for various classes of metric space and disproven for a few others, most notably expanders and warped cones \cite{HLS02,sawicki_warped_2017}. The interested reader is referred to \cite{aparicio_baum-connes_2019} and references therein for more information about the conjecture.
\end{rmk}

We already remarked that elements of the (uniform) Roe algebra are quasi\=/local operators. It follows that $P_X$ is quasi\=/local if $X$ consists of a sequence of expanders. Furthermore, as already anticipated in the introduction, we have the following:

\begin{proposition}[{\cite[Theorem~B]{intro}}]
Let $\{X_n\}_{n\in \N}$ be a sequence of finite metric spaces with $|X_n|\rightarrow \infty$ and let $X$ be their coarse disjoint union. Then the associated averaging projection $P_X$ is quasi-local \emph{if and only if} $\{X_n\}_{n\in \N}$ is a sequence of asymptotic expanders.
\end{proposition}

\section{Structure theorem for asymptotic expanders}\label{sec:structure}

In this section, we establish the fundamental structure theorem for asymptotic expanders, showing that they can be ``exhausted'' by a sequence of spaces with expander behaviour. We start with the graph case, which may be more intuitive, and then move on to the general case.

\subsection{Graph case}
As mentioned in the introduction, the key observation that will enable us to prove the structure result is the following simple lemma:

 \begin{lem}\label{lem:complement.of.maximal.F}
  Let $X$ be a finite graph and fix $c>0$. If $F\subset X$ is a maximal $c$\=/\fol set (with respect to inclusion), then for every non\=/empty subset $A\subset X\smallsetminus F$ such that $\abs{A}\leq \frac{1}{2}\abs{X}-\abs{F}$ we have $\abs{\partial A\smallsetminus F}>c\abs{A}$.
 \end{lem}
\begin{proof}
 Let $Y\coloneqq X\smallsetminus F$ and let $A\subset Y$ with $0<\abs{A}\leq \frac{1}{2}\abs{X}-\abs{F}$. Since $F$ is a maximal $c$\=/\fol set and $\abs{A\sqcup F}\leq \frac{1}{2}\abs{X}$, we deduce that $\abs{\partial^X (A\sqcup F)}> c\abs{A\sqcup F}=c(\abs{F}+\abs{A})$. On the other hand, we have:
 \[
  \abs{\partial^X (A\sqcup F)}\leq \abs{\partial^X F}+\abs{\partial^Y A}\leq c\abs{F}+\abs{\partial^YA}.
 \]
 Putting the two inequalities together we obtain that $\abs{\partial^YA} > c\abs{A}$, as desired.
\end{proof}

\begin{cor}\label{cor:complement.of.maximal.F}
  Let $F\subset X$ be as in Lemma~\ref{lem:complement.of.maximal.F} and let $Y\coloneqq X\smallsetminus F$.
  If $\abs{F}\leq\alpha\abs{X}$, then any non\=/empty $c$\=/\fol set $A$ for the graph $Y$ must satisfy $\abs{A}>\frac{1-2\alpha}{2}\abs{X}$.
\end{cor}
\begin{proof}
 This follows immediately from Lemma~\ref{lem:complement.of.maximal.F} because
 \[
  \frac{1-2\alpha}{2}\abs{X}
  =\frac{1}{2}\abs{X}-\alpha\abs{X}<\frac{1}{2}\abs{X}-\abs{F}
 \]
 and hence a set with $\abs{A}\leq\frac{1-2\alpha}{2}\abs{X}$ cannot be a $c$\=/\fol set for $Y$.
\end{proof}

Corollary~\ref{cor:complement.of.maximal.F} really shines when one is provided with maximal \fol sets of very small size. In fact, when this  is the case, \fol sets in the complement $Y=X\smallsetminus F$ must be very large and are hence easier to control. This is the basic strategy at the heart of the Structure Theorem for asymptotic expander graphs. Before delving into the proof it is convenient to introduce some extra notation.

\begin{de}\label{de:exhaustion}
 Let $\{X_n\}_n$ be a sequence of finite graphs with uniformly bounded degrees and $\abs{X_n}\to\infty$. We say that it admits a \emph{uniform exhaustion by expanders} if there exist sequences $\{\alpha_k\}_{k\in \N}, \{c_k\}_{k\in \N}$ such that $c_k>0$ and $\alpha_k\to 0$, and subsets $Y_{n,k}\subseteq X_n$ such that $\{Y_{n,k}\}_n$ is a $c_k$\=/expander and $\abs{Y_{n,k}}\geq (1-\alpha_k)\abs{X_n}$ for every $n,k\in\NN$.
\end{de}

\begin{rmk}
  A few comments:
  \begin{enumerate}
  \item For fixed $k$, the sequence $\braces{Y_{n,k}}_{n\in\NN}$ is a sequence of expander graphs.
  \item For fixed $k$, the inclusions $Y_{n,k}\hookrightarrow X_n$ need not be uniformly coarse embeddings (but they are ``weak embeddings''. See the discussion in Section~\ref{sec:non.coarse.embeddability}).
  \item For fixed $n$, there is a $k$ large enough so that $\alpha_k<\frac{1}{\abs{X_n}}$. In that case $Y_{n,k}$ must be the whole $X_n$.
  \item Let $Y,Y'\subset X$ be $c$\=/expanders such that $\abs{Y},\abs{Y'}\geq (1-\alpha)\abs{X}$. When $\alpha$ is small it is easy to show that $Y\cup Y'$ is going to be a $c'$\=/expander for some constant $c'$ depending only on $c$ and $\alpha$. Using this observation and considering unions of the $Y_{n,k}$, one can always assume that $Y_{n,k}\subseteq Y_{n,k'}$ for every $k'>k$. This justifies the usage of the term ``exhaustion''.
  \end{enumerate}
\end{rmk}

Let $\{X_n\}_{n\in \N}$ be a sequence of graphs with uniformly bounded degree and $|X_n| \to \infty$.
Since the $X_n$ have uniformly bounded degree, it follows from Lemma~\ref{lem: equiv def for asymptotic expanders(graph)} that $\braces{X_n}_{n\in\NN}$ is a sequence of asymptotic expanders if and only if there is a function $\fnc \colon(0,\frac{1}{2}]\to [0,\infty)$ such that for every $\alpha\in (0,\frac 12]$ the cardinality of any $\fnc(\alpha)$\=/\fol set in $X_n$ is at most $\alpha \abs{X_n}$ (in the notation of Lemma~\ref{lem: equiv def for asymptotic expanders(graph)}, $\fnc$ can be defined as $\fnc(\alpha)\coloneqq c_\alpha$). Since $c$\=/\fol sets are $c'$\=/\fol sets for every $c'>c$, we can assume that the function $\fnc$ is an increasing function.

For every $c>0$, let $\CF_{c,n}$ be the family of $c$\=/\fol sets in $X_n$, \emph{i.e.},
 \[
  \CF_{c,n}\coloneqq\Bigbraces{A\subset X_n\Bigmid \abs{A}\leq\frac{1}{2}\abs{X_n},\ \abs{\partial A}\leq c\abs{A}}.
 \]
By our choice of conventions, $\CF_{c,n}$ is always non\=/empty, as it contains the empty set. Further, since $X_n$ is a finite graph, $\CF_{c,n}$ always admits some (possibly more than one) element that is maximal with respect to inclusion.
Note that, if $\braces{X_n}_{n\in\NN}$ is a sequence of asymptotic expanders, then every $F_n\in\CF_{\fnc(\alpha),n}$ has cardinality at most $\abs{F_n}\leq\alpha\abs{X_n}$. We are now ready to prove the following:

\begin{thm}[Structure Theorem for asymptotic expander graphs]\label{thm:structure.of.as.exp.}
 Let $\braces{X_n}_{n\in\NN}$ be a sequence of finite connected graphs with uniformly bounded degrees and increasing cardinalities. Then it is a sequence of asymptotic expanders if and only if it admits a uniform exhaustion by expanders.
\end{thm}

\begin{proof}
 Assume that $\braces{X_n}_{n\in\NN}$ is a sequence of asymptotic expanders with function $\fnc\colon(0,\frac{1}{2}]\to (0,\infty)$.
 If $\fnc$ is bounded away from $0$ then $\braces{X_n}_{n\in\NN}$ is a sequence of expanders and there is nothing to prove. We thus assume that $\fnc(\alpha)\to 0$ as $\alpha\to 0$. Fix $\alpha$ (small) and let $F_{n,\alpha}$ be a maximal $\fnc(\alpha)$\=/\fol set (\emph{i.e.} a maximal element in $\CF_{\fnc(\alpha),n}$).
 By definition, $\abs{F_{n,\alpha}}\leq \alpha\abs{X_n}$ and hence $Y_{n,\alpha}\coloneqq X_n\smallsetminus F_{n,\alpha}$ has cardinality at least $(1-\alpha)\abs{X_n}$. By Corollary~\ref{cor:complement.of.maximal.F} every non\=/empty subset $A\subset Y_{n,\alpha}$ with $|A| \leq(\frac{1}{2}-\alpha)|X_n|$ is not a $\fnc(\alpha)$\=/\fol set.

 On the other hand, if $A\subset Y_{n,\alpha}$ is such that $(\frac{1}{2}-\alpha)\abs{X_n}<\abs{A}\leq\frac{1}{2}\abs{Y_{n,\alpha}}$ then we have:
 \[
  \abs{\partial^{Y_{n,\alpha}}A}\geq\abs{\partial^{X_n}A}-D\abs{\partial^{X_n} F_{n,\alpha}}
  \geq \fnc\Bigparen{\frac{1}{2}-\alpha}\abs{A}-D\fnc(\alpha)\abs{F_{n,\alpha}}
\]
where $D$ is the bound on the degrees. Whence
 \begin{equation}\label{eq:big.bdary.structure.thm}
  \abs{\partial^{Y_{n,\alpha}}A}
  \geq \Bigbrack{\fnc\Bigparen{\frac{1}{2}-\alpha}-D\fnc(\alpha)\frac{2\alpha}{1-2\alpha}}\abs{A}.
 \end{equation}

 Since $\fnc$ is a strictly positive non\=/decreasing function, the function $\fnc\Bigparen{\frac{1}{2}-\alpha}-D\fnc(\alpha)\frac{2\alpha}{1-2\alpha}$ increases as $\alpha$ gets smaller and its limit as $\alpha$ goes to $0$ is strictly positive. On the  other hand, we are assuming that $\fnc(\alpha)$ converges to $0$ as $\alpha\to 0$ and hence there exists an $\alpha_0$ small enough such that for every $\alpha<\alpha_0$ we have
 \[
  \fnc\Bigparen{\frac{1}{2}-\alpha}-D\fnc(\alpha)\frac{2\alpha}{1-2\alpha}
  >\fnc(\alpha_0)\geq \fnc(\alpha).
 \]
Hence, we deduce from \eqref{eq:big.bdary.structure.thm} that for any given $\alpha<\alpha_0$ the graphs $Y_{n,\alpha}$\=/are $\fnc(\alpha)$\=/expanders.

Choosing a decreasing sequence $\{\alpha_k\}_{k \in \N}$ tending to 0 with $\alpha_k<\alpha_0$ and letting $Y_{n,k}\coloneqq Y_{n,\alpha_k}$ yields a uniform exhaustion by expanders.

 \

 For the converse implication, assume that $(Y_{n,k})_{n,k\in\NN}$ is a uniform exhaustion by expanders for $\braces{X_n}_{n\in\NN}$. Then for every $\alpha$ there is a $k$ large enough so that $\abs{Y_{n,k}}>(1-\frac{\alpha}{2})\abs{X_n}$ for every $n$.

 Since $Y_{n,k}$ is a $c_k$\=/expander, for every subset $A\subset X_n$ such that $\alpha\abs{X_n} \leq \abs{A} \leq \frac{1}{2}\abs{X_n}$ and $\abs{A\cap Y_{n,k}}\leq\frac{1}{2}\abs{Y_{n,k}}$ we have
 \[
  \abs{\partial^{X_n}A}\geq \abs{\partial^{Y_{n,k}}(A\cap Y_{n,k})}\geq c_k\abs{A\cap Y_{n,k}}\geq \frac{c_k}{2}\abs{A},
 \]
 where the last inequality follows from the fact that $|Y_{n,k}| > |X_n| - \frac{1}{2}|A|$ and hence $|A \setminus Y_{n,k}| \leq |A \cap Y_{n,k}|$.
 It only remains to deal with the case of $A\subset X_n$ such that $\alpha\abs{X_n} \leq \abs{A} \leq \frac{1}{2}\abs{X_n}$ and $\frac{1}{2}\abs{Y_{n,k}}<\abs{A\cap Y_{n,k}}$. In this case we look at the complement in $Y$ and we have
 \[
  \abs{\partial^{X_n}A}\geq \abs{\partial^{Y_{n,k}}(A\cap Y_{n,k})}
  \geq \frac{1}{D}\abs{\partial^{Y_{n,k}}(Y_{n,k}\smallsetminus A)}
  > \frac{c_k}{D}\abs{Y_{n,k}\smallsetminus A}
  \geq \frac{c_k}{D}\Bigparen{1-\frac{\alpha}{2}-\frac{1}{2}}\abs{X_n}
  \]
 and hence $\abs{\partial^{X_n}A}\geq \frac{c_k}{D}(1-\alpha)\abs{A}$. Letting $c_\alpha\coloneqq\min\bigbraces{\frac{c_k}{2},\frac{c_k}{D}(1-\alpha)}$ satisfies the requirements of the definition of asymptotic expander graphs.
\end{proof}

\begin{rmk}
Theorem~\ref{thm:structure.of.as.exp.} holds with the same proof when $\braces{X_n}_{n\in\NN}$ is a sequence of finite metric spaces that are \emph{uniformly quasi\=/geodesic} (\emph{i.e.}, so that any two points can be joined by an $(L,A)$\=/quasi\=/geodesic for some fixed $L$ and $A$) and have uniformly bounded geometry.
\end{rmk}

\subsection{General case}

The proof of Theorem~\ref{thm:structure.of.as.exp.} made substantial use of the characterisation of asymptotic expanders in terms of condition (AsEx''). In the general setting of families of metric spaces, condition (AsEx'') ceases to be equivalent to the definition of asymptotic expanders and we are thus forced to revert to using condition (AsEx') instead. That is, we use that  a sequence of finite metric spaces $\{X_n\}_{n \in \N}$ with $|X_n| \to \infty$ is a sequence of asymptotic expanders if and only if there exists some $c>0$, such that for any $\alpha\in (0,\frac{1}{2}]$, there exists $R_\alpha>0$ such that for any $n\in \N$ and $A \subseteq X_n$ with $\alpha|X_n| \leq |A| \leq \frac{1}{2}|X_n|$, we have $|\partial_{R_\alpha} A| > c|A|$.

The essential difference between conditions (AsEx') and (AsEx'') is the order of the quantifiers: in (AsEx'') there is a fixed radius $R=1$ so that for every $\alpha$ there is a $c_\alpha$ such that $\abs{\partial A}> c_\alpha\abs{A}$; while in (AsEx') there is a fixed $c>0$ so that for every $\alpha$ there is a radius $R_\alpha$ such that $\abs{\partial_{R_\alpha} A}> c\abs{A}$ (with the necessary restrictions to the subsets $A$ under consideration). This difference is reflected in the statement of the Strucure Theorem: the analogue of Theorem~\ref{thm:structure.of.as.exp.} in the general metric case becomes the following:

\begin{thm}[Structure Theorem for general asymptotic expanders]\label{thm:structure.of.as.exp.general.case}
Let $\{X_n\}_{n \in \N}$ be a sequence of finite metric spaces with $|X_n| \to \infty$. Then the following are equivalent:
\begin{enumerate}
  \item $\{X_n\}_{n \in \N}$ is a sequence of asymptotic expanders;
  \item there exist $c>0$, a positive decreasing sequence $\{\alpha_k\}_{k \in \N}$ with $\alpha_k \to 0$, a positive sequence $\{R_k\}_{k\in \N}$, and for each $n \in \N$ a sequence of subsets $\{Y_{n,k}\}_{k\in \N}$ in $X_n$, satisfying: for each $n,k\in \N$,
  \begin{itemize}
    \item $|Y_{n,k}| \geq (1-\alpha_k)|X_n|$;
    \item for each non-empty $A \subseteq Y_{n,k}$ with $|A| \leq \frac{1}{2}|Y_{n,k}|$, we have $|\partial_{R_k}^{Y_{n,k}} A| > c|A|$.
  \end{itemize}
\end{enumerate}
\end{thm}

\begin{proof}

\emph{``(1) $\Rightarrow$ (2)''}: A large portion of the proof is very similar to the proof of Theorem~\ref{thm:structure.of.as.exp.}:
we consider for every $c,R>0$ and $n \in \N$ the family $\CF_{c,R,n}$ of subsets in $X_n$ given by
 \[
  \CF_{c,R,n}\coloneqq\Bigbraces{A\subset X_n\Bigmid \abs{A}\leq\frac{1}{2}\abs{X_n},\ \abs{\partial_R A}\leq c\abs{A}}.
 \]
Again, $\CF_{c,R,n}$ is always non\=/empty, as it contains the empty set and it always admits some (possibly more than one) element that is maximal with respect to inclusion.

We can now fix a constant $c>0$ as in condition (AsEx') and for each $\alpha\in (0,\frac{1}{2}]$ and $n \in \N$, we consider the set $\CF_{c,R_\alpha,n}$ defined above. Take $F_{n,\alpha}$ to be a maximal element in $\CF_{c,R_\alpha,n}$, and set $Y_{n,\alpha} \coloneqq X_n \setminus F_{n,\alpha}$. It follows by the assumption that $|F_{n,\alpha}| \leq \alpha|X_n|$ for each $n \in \N$.

For any non-empty $A \subseteq Y_{n,\alpha}$ with $|A| \leq \frac{1}{2}|X_n| - |F_{n,\alpha}|$, clearly $|A \sqcup F_{n,\alpha}| \leq \frac{1}{2}|X_n|$. We also have:
\[
|\partial^{X_n}_{R_\alpha} (A \sqcup F_{n,\alpha})| \leq |\partial^{X_n}_{R_\alpha} F_{n,\alpha}| + |\partial^{Y_{n,\alpha}}_{R_\alpha} A| \leq c|F_{n,\alpha}| + |\partial^{Y_{n,\alpha}}_{R_\alpha} A|.
\]
On the other hand, by the maximality of $F_{n,\alpha}$ we have
\[
|\partial^{X_n}_{R_\alpha} (A \sqcup F_{n,\alpha})| > c|A \sqcup F_{n,\alpha}| = c|A| + c|F_{n,\alpha}|.
\]
Combining them together, we have $|\partial^{Y_{n,\alpha}}_{R_\alpha} A| > c|A|$.

For any non-empty $A \subseteq Y_{n,\alpha}$ with $\frac{1}{2}|X_n| - |F_{n,\alpha}| < |A| \leq  \frac 12 |Y_{n,\alpha}|$, we take a subset $A' \subseteq A$ with $|A'|=\frac{1}{2}|X_n| - |F_{n,\alpha}|$.
Then from the above analysis, we have
\[
|\partial^{Y_{n,\alpha}}_{R_\alpha} A| \geq |\partial^{Y_{n,\alpha}}_{R_\alpha} A'| - |A \setminus A'| > c|A'| - \frac{|F_{n,\alpha}|}{2} = \frac{c}{2}|X_n| - (c+\frac{1}{2})|F_{n,\alpha}| \geq [c - (2c+1) \alpha] |A|.
\]

Hence for $\alpha\leq \frac{c}{2(2c+1)}$, we have $|\partial^{Y_{n,\alpha}}_{R_\alpha} A| > \frac{c}{2}|A|$.\footnote{This is the main point where this proof diverges from the proof of Theorem~\ref{thm:structure.of.as.exp.}. In fact, when choosing $\alpha\leq \frac{c}{2(2c+1)}$ we are making use of the fact that the constant $c$ is fixed. On the other hand, in Theorem~\ref{thm:structure.of.as.exp.} $c$ depends on $\alpha$.
}

Combining the above two paragraphs together, we showed: if $\alpha\leq \frac{c}{2(2c+1)}$, then for any non-empty $A \subseteq Y_{n,\alpha}$ with $|A| \leq \frac{1}{2}|Y_{n,\alpha}|$, we have $|\partial^{Y_{n,\alpha}}_{R_\alpha} A| > \frac{c}{2}|A|$. Therefore, we take a positive decreasing sequence $\{\alpha_k\}_{k \in \N}$ with $\sup_{k\in \NN}\alpha_k \leq \frac{c}{2(2c+1)}$ and $\alpha_k \to 0$, and take $R_k\coloneqq R_{\alpha_k}$. For each $k,n \in \N$, we take $Y_{n,k}\coloneqq Y_{n,\alpha_k}$. Then clearly condition (2) holds for $\frac{c}{2}$.

~\

\emph{``(2) $\Rightarrow$ (1)''}: We assume condition (2) holds with the constants therein. Given $\alpha>0$, we take $\alpha_k< \min\{\frac{\alpha}{2}, \frac{c}{2(1+c)}\}$. Given any $n \in \N$ and a non-empty subset $A \subseteq  X_n$ with $\alpha |X_n| \leq |A| \leq \frac{1}{2}|X_n|$, we consider the following two cases:

If $|A \cap Y_{n,k}| \leq \frac{1}{2}|Y_{n,k}|$: by assumption, we have $|\partial^{X_n}_{R_k} A| \geq |\partial^{Y_{n,k}}_{R_k} (A \cap Y_{n,k})| > c|A \cap Y_{n,k}|$. Furthermore, we have
\[
|A\cap Y_{n,k}| \geq |A| - |X_n \setminus Y_{n,k}| \geq |A| - \alpha_k|X_n| \geq |A| - \frac{\alpha}{2}|X_n| \geq \frac{1}{2}|A|.
\]
Hence combining them together, we obtain: $|\partial^{X_n}_{R_k} A| > \frac{c}{2}|A|$.

If $|A \cap Y_{n,k}| > \frac{1}{2}|Y_{n,k}|$: we take a subset $A' \subseteq A \cap Y_{n,k}$ with $|A'|=\frac{1}{2}|Y_{n,k}|$. Then combining the above analysis, we have:
\begin{align*}
|\partial^{X_n}_{R_k} A| \geq & |\partial^{Y_{n,k}}_{R_k} (A \cap Y_{n,k})| \geq |\partial^{Y_{n,k}}_{R_k} A'| - |(A \cap Y_{n,k}) \setminus A'|
 >  c|A'| - \big( \frac{1}{2}|X_n| - |A'| \big) \\
 = & \frac{1+c}{2}|Y_{n,k}| - \frac{1}{2}|X_n|
 \geq  \big[ \frac{1+c}{2}\cdot (1-\alpha_k) -\frac{1}{2} \big] |X_n| \geq [c-(1+c)\alpha_k]|A|
 \geq  \frac{c}{2}|A|.
\end{align*}

Combining the above two cases together, we have shown condition (1) holds.
\end{proof}

\section{Lack of coarse embeddings}\label{sec:non.coarse.embeddability}

This section is concerned with coarse embeddings into Banach spaces. Since we are only interested in metric properties, it is natural to focus on real Banach spaces. In particular, we will treat complex Banach spaces as if they were real by ignoring the extra complex structure (\emph{i.e.} $\CCC$ is isometric to $\RR^2$). Before stating the main result we need to introduce some terminology.

Given a (possibly infinite) measure space $(\Omega,\mu)$ and a Banach space $E$, we denote by $L^p(\Omega;E)$ the set of measurable functions $f\colon\Omega\to E$ such that $\norm{f(\omega)}^p$ is integrable. Taking the $p$-th root of this integral defines a norm that makes $L^p(\Omega;E)$ into a Banach space. This is the space of \emph{$p$\=/Bochner integrable functions}. We denote by $\ell^p_\RR$, $\ell^p_\CCC$, $L^p_\RR$ and $L^p_\CCC$ the spaces $L^p(\NN,\RR)$, $L^p(\NN,\CCC)$, $L^p([0,1],\RR)$ and $L^p([0,1],\CCC)$ respectively.

We will make use of the notion of \emph{uniformly curved Banach spaces} as a black box, and just say that they were defined by Pisier in \cite{pisier_complex_2010}. In a sense, this is an optimal class of (complex) Banach spaces for which one can immediately produce good bounds on operator norms using complex interpolation.

Finally, two Banach spaces $E_1$ and $E_2$ are \emph{sphere equivalent} if their unit spheres $S(E_1)$ and $S(E_2)$ are homeomorphic via a homeomorphism $\phi$ such that both $\phi$ and $\phi^{-1}$ are \emph{uniformly continuous}.

Let $\CC_0$ be the smallest class of Banach spaces such that:
\begin{itemize}
 \item $\RR$, $\CCC$, $\ell^p_\RR$, $\ell^p_\CCC$, $L^p_\RR$ and $L^p_\CCC$ belong to $\CC_0$ for every $p\in[1,\infty)$,
 \item uniformly curved Banach spaces are in $\CC_0$,
 \item if $E\in\CC_0$ then $L^p(\Omega;E)\in\CC_0$ for every $(\Omega,\mu)$ and $p\in[1,\infty)$,
 \item if $E_1\in\CC_0$ and $E_2$ is sphere equivalent to $E_1$, then $E_2\in\CC_0$.
\end{itemize}

\begin{rmk}
 The above definition is highly redundant. For example, the spaces $\ell^p_\RR$, $\ell^p_\CCC$, $L^p_\RR$ and $L^p_\CCC$ with $p\in[1,\infty)$ are all sphere\=/equivalent to the separable infinite-dimensional Hilbert space via the Mazur map (\cite[Chapter 10]{benyamini_geometric_1998} and \cite{cheng2016sphere}).
\end{rmk}

The aim of this section is to prove the following theorem, which both answers \cite[Question~7.3]{intro} and generalizes \cite[Proposition~7.4]{intro}.

\begin{thm}\label{thm: non CE general case}
Let $\{X_n\}_{n\in \N}$ be a sequence of finite metric spaces with $|X_n|\rightarrow \infty$ and uniformly bounded geometry, and $X$ be their coarse disjoint union. If $\{X_n\}_{n\in \N}$ is a sequence of asymptotic expanders, then $X$ cannot be coarsely embedded into any Banach space in the class $\CC_0$.
\end{thm}

Since Theorem~\ref{thm:structure.of.as.exp.general.case} is established, the proof of Theorem~\ref{thm: non CE general case} is more or less folklore. In fact, it follows from Theorem~\ref{thm:structure.of.as.exp.general.case} that the space $X$ contains weakly embedded expanders and it is well\=/known to experts that this is an obstruction to the existence of such coarse embeddings.
For the convenience of the reader, we record here a more detailed proof. The techniques are by now fairly standard, and are used \emph{e.g.} in \cite{arzhantseva_relative_2015,de_laat_superexpanders_2019,gromov_random_2003,matousek_embedding_1997,mendel_nonlinear_2014,mimura2014sphere,tessera_coarse_2009}. For a different point of view on the theory of embeddings of metric spaces into Banach spaces, we also recommend the survey \cite{baudier2014metric}.

Let $E$ be a (real) Banach space and fix $p\in[1,\infty)$. Following  \cite{mimura2014sphere}, we say that the \emph{$(E,p)$\=/spectral gap} of a finite graph $\CG$ is
\begin{equation}\label{eq:p-spectral.gap}
 \lambda_1(\CG; E,p)\coloneqq \frac{1}{2} \inf
 \frac{\sum_{x\thicksim y}\norm{f(x)-f(y)}^p}{\sum_{x\in\CG}\norm{f(x)-m(f)}^p}
\end{equation}
where $x\sim y$ means that $x,y\in\CG$ are joined by an edge, $m(f)=\frac{1}{\abs{\CG}}\sum_{x\in\CG}f(x)$ is the mean value of $f$ and the infimum is taken over all non\=/constant functions $f\colon \CG\to E$.
Equivalently, $\lambda_1(\CG; E,p)$ is the largest constant $\epsilon$ so that $\CG$ satisfies an \emph{$E$\=/valued (linear) $p$\=/Poincaré inequality}
\begin{equation}\label{eq:p-poincare}
  \frac 12 \sum_{x\thicksim y} \norm{f(x)-f(y)}^p \geq \epsilon \sum_{x\in \CG}\norm{f(x)-m(f)}^p
\end{equation}
for every $f\colon\CG\to E$.

\begin{rem}
There is a competing notion of Poincaré inequality. That is, we say that $\CG$ satisfies an $E$\=/valued \emph{metric} $p$\=/Poincaré inequality if
\[
  \frac{1}{\abs{\text{edges in $\CG$}} }\Bigparen{\frac 12 \sum_{x\thicksim y} \norm{f(x)-f(y)}^p }
  \geq \frac{\epsilon'}{\abs{\CG}^2} \sum_{x,y\in \CG}\norm{f(x)-f(y)}^p.
\]
Using $\norm{f(x)-f(y)}^p\leq\paren{\norm{f(x)}+\norm{f(y)}}^p$ it is easy to show that linear $p$\=/Poincaré inequalities imply metric $p$\=/Poincaré inequalities (possibly with a different constant). On the other hand, the converse implication is not clear in general.
\endgraf
In order to find obstructions to coarse embeddability into Banach spaces, it is sufficient to have metric Poincaré inequalities. Further, metric Poincaré inequalities can be defined also when the target space $E$ is a general metric space (as opposed to Banach space). For this reason they are used \emph{e.g.} in \cite{mendel_nonlinear_2014}.\endgraf
As a last remark, it can be shown that linear $p$\=/Poincaré inequalities are well\=/behaved under change of $p$ (see below). It is not clear whether the same holds for metric $p$\=/Poincaré inequalities.
\end{rem}

The importance of Poincaré inequalities in the context of coarse embedding is given by the following observation: if $\CG$ has degree bounded by $D$ and satisfies an $E$\=/valued (linear) $p$\=/Poincaré inequality with constant $\epsilon$, then for every $L$\=/Lipschitz function $f\colon\CG\to E$ with $m(f)=0$ we have
\begin{equation}\label{eq:bound.on.Lipschitz}
 \frac{2\epsilon}{\abs{\CG}}\sum_{x\in\CG}\norm{f(x)}^p\leq\frac{1}{\abs{\CG}}\sum_{x\thicksim y}L \leq DL^p.
\end{equation}
That is, there is an upper bound on the expected value of (the $p$\=/th power of) $\norm{f(x)}$ which does not depend on $\abs{\CG}$. If one knows that the expected distance between pairs of vertices of $\CG$ is very large, then \eqref{eq:bound.on.Lipschitz} implies that any $L$\=/Lipschitz function to $E$ is far from being bi-Lipschitz.

Back to Theorem~\ref{thm: non CE general case}, if $\{X_n\}$ is a sequence of asymptotic expanders it follows from Theorem~\ref{thm:structure.of.as.exp.general.case} that there exist $c>0$, $R>0$, and a sequence of subspaces $\{Y_n \subseteq X_n\}_{n\in \N}$ with $|Y_n| \geq \frac{3}{4}|X_n|$ and such $|\partial_{R}^{Y_n} A| > c|A|$ for every $A \subseteq Y_n$ with $|A| \leq \frac{1}{2}|Y_n|$.
For each $n$, we consider the finite graph $\mathcal{G}_n$ whose vertex set is $Y_n$, and where $x,y\in Y_n$ are connected by an edge if and only if $d(x,y) \leq R$. Since the $X_n$ have uniformly bounded geometry, the graphs $\CG_n$ have uniformly bounded degree (and hence they are expanders).

Note that the inclusion $Y_n\subset X_n$ need not induce a coarse embedding $\CG_n\hookrightarrow X_n$, it is therefore not possible to directly conclude that whenever the $\CG_n$ do not uniformly coarsely embed into a Banach space $E$ neither do the $X_n$.
On the other hand, the existence of Poincaré inequalities for $\CG_n$ does provide us with an obstruction to the existence of uniform coarse embeddings of the $X_n$.

In fact, if there exists a $p\in[1,\infty)$ and an $\epsilon>0$ such that $\lambda_1(\CG_n; E,p)>\epsilon$ for every $n\in\NN$, then $X$ cannot coarsely embed into $E$.
This can be shown with the following simple argument: assume by contradiction that there exist coarse embeddings $f_n\colon X_n\to E$ with uniform control functions $\rho_-$ and $\rho_+$.
Then the restriction of $f_n$ to $\CG_n$ is $\rho_+(R)$\=/Lipschitz. Subtracting the mean if necessary, we can further assume that $m(f_n)=0$ for all $n$.
Let $D$ be the uniform bound on the degrees of $\CG_n$, it follows from \ref{eq:bound.on.Lipschitz} that the average $\frac{1}{\abs{\CG_n}}\sum_{x\in Y_n}\norm{f(x)}^p$ is uniformly bounded by $C\coloneqq \frac{D\rho_+(R)^{p}}{2\epsilon}$. In particular, $\norm{f(x)}^p\leq 2C$ for at least half of the vertices $x\in \CG_n$. Therefore, at least half of the vertices of $\CG_n$ are within distance $2\rho_-^{-1}(2C)$ from one another. On the other hand, the $\CG_n$ have degree bounded by $D$ and hence the ball of radius $r$ in $\CG_n$ has cardinality at most $D^{r+1}$. This brings us to a contradiction as $\abs{\CG_n}\to\infty$.

\begin{rmk}
 The above argument is nothing but the proof of the well\=/known fact that Poincaré inequalities provide obstructions to the existence of weak embeddings (see \cite{arzhantseva_relative_2015, gromov_random_2003}).
If $X=\bigsqcup_{n\in\NN} X_n$ is as in Theorem~\ref{thm: non CE general case}, it follows from the bounded geometry assumption that the inclusions $Y_n\subseteq X_n$ yield a weak embedding $\CG_n\to X$. Further, if $f\colon X\to E$ was a coarse embedding, then the compositions $\CG_n\to X\to E$ would again be a weak embedding. The above argument proves that such a weak embedding cannot exist.
\end{rmk}

As already remarked, the graphs $\CG_n$ appearing in the above argument are a sequence of expanders.
Therefore, all that remains to do to prove Theorem~\ref{thm: non CE general case} is to show that expanders satisfy the relevant Banach\=/valued (linear) $p$\=/Poincaré inequalities.
For sake of brevity, we will use Mimura's notation and say that a sequence of graphs $\CG_n$ is a sequence of \emph{$(E,p)$\=/anders}\footnote{%
In \cite{mimura2014sphere}, Banach spaces are usually denoted by $X$. The nomenclature $(X,p)$\=/ander is a pun.}
if they have uniformly bounded degree, $\abs{\CG_n}\to\infty$ and there is a $\epsilon>0$ such that $\lambda_1(\CG_n;E,p)>\epsilon$ for every $n$. Theorem~\ref{thm: non CE general case} will be proved when we show that expanders are $(E,p)$\=/anders for every $E\in\CC_0$ and some $p\in[1,\infty)$. The following discussion is basically an illustration of \cite[Corollary C]{mimura2014sphere} and \cite[Section 4]{cheng2016sphere}.

It is classical that a sequence of graphs is a sequence of expanders if and only if they are $(\RR,2)$\=/anders (the proof is fairly simple, and it can be extrapolated \emph{e.g.} from the arguments in  \cite{matousek_embedding_1997}). It is also easy to prove the analogous statement for $(\RR,1)$\=/anders.

Matou\v{s}ek proved in \cite{matousek_embedding_1997} that being $(\RR,p)$\=/anders does not depend on $p$ (his strategy is often called Matou\v{s}ek extrapolation \cite{mendel_nonlinear_2014,mimura2014sphere}).
This was greatly generalized in \cite{mimura2014sphere}, and finally Cheng proved in \cite[Theorem 4.9]{cheng2016sphere} that for any fixed Banach space $E$ the property of being an $(E,p)$\=/ander does not depend on $p\in[1,\infty)$. In particular, one could in principle avoid considering $p$\=/Poincaré inequalities and just focus on, say, $2$\=/Poincaré inequalities. Yet, considering general $p$ has its own advantages. In fact, noting that
\[
 \sum_{x\in\CG}\norm{f_x}^p_{L^p}
 =\sum_{x\in\CG}\int_{\Omega}\abs{f_x(\omega)}^p d\mu(\omega)
 =\int_{\Omega}\sum_{x\in\CG}\abs{f_x(\omega)}^p d\mu(\omega),
\]
it is easy to show that $\lambda_1(\CG_n;\RR,p)=\lambda_1(\CG_n;L^p_\RR,p)$. This proves that asymptotic expanders cannot coarsely embed into $\ell^p_\RR$ and $L^p_\RR$.

More generally, the same argument applies to every $p$\=/Bochner space and shows that $\lambda_1(\CG_n;E,p)=\lambda_1(\CG_n;L^p(\Omega;E)_\RR,p)$ for every measure space $(\Omega,\mu)$ and Banach space $E$. In particular, if the $\CG_n$ are $(E,p)$\=/anders then they are also $(L^p(\Omega;E),p)$\=/anders. Since $\CCC$ is isometric to $\RR^2\subset\ell^2_\RR$, this also implies that asymptotic expanders do not coarsely embed into $\ell^p_\CCC$ and $L^p_\CCC$.

The fact that expanders are $(E,2)$\=/anders for every uniformly curved Banach space $E$ follows from the discussion in \cite[Chapter 3]{pisier_complex_2010}.

All that remains to show is that the property of being an $(E,p)$\=/ander is stable under sphere equivalence of Banach spaces. This is proved in \cite[Theorem A]{mimura2014sphere} and \cite[Theorem 4.9]{cheng2016sphere} (see also \cite[Appendix A]{ozawa_note_2004}). This concludes the proof of Theorem~\ref{thm: non CE general case}.

\begin{rmk}\label{rmk:coarse.image.of.asymptotic.expander}
 It is interesting to note that metric spaces that \emph{coarsely} contain bounded geometry asymptotic expanders must also \emph{isometrically} contain such asymptotic expanders. Namely, if $\braces{X_n}_{n\in\NN}$ are asymptotic expanders with uniformly bounded geometry and $f_n\colon X_n\to Y$ are coarse embeddings with uniform control functions, then the sequence of sets $\braces{f_n(X_n)}_{n\in\NN}$ equipped with the subset metric have uniformly bounded geometry and are uniformly coarsely equivalent to the $X_n$. Since asymptotic expansion for spaces with bounded geometry is preserved under coarse equivalence (\cite[Theorem 3.11]{intro}), we deduce that $\braces{f_n(X_n)}_{n\in\NN}$ is a sequence of asymptotic expanders that embed isometrically in $Y$.  
\end{rmk}

\section{Some examples}\label{sec:examples}
In this section we are mostly interested in producing examples of families of graphs with uniformly bounded degree that are asymptotic expanders but non-expanders.

\noindent\emph{A Foreword.} The examples that we are going to discuss are obtained via simple geometric tricks and they are more of a proof-of-concept rather then actual examples.
The bad news is that Theorem~\ref{thm:structure.of.as.exp.} implies that, in some sense, every asymptotic expanders can be obtained from actual expanders by using similar tricks.
The good news is that asymptotic expanders are not a mere artificial construction, but they appear quite naturally.
A prime ``natural'' source of examples of asymptotic expanders that are not expanders is provided in \cite{li2021asymptotic}, where it is proved that families of graphs approximating a \emph{strongly ergodic} action are asymptotic expanders. On the other hand, it is known that such graphs are a sequence of expanders if and only if the action is \emph{expanding in measure} \cite{vigolo_measure_2019}. In particular, graphs obtained by approximating a measure\=/preserving action that is strongly ergodic but does not have a spectral gap, are asymptotic expanders without being expanders.
One explicit example of such an action is given in \cite{schmidtAmenabilityKazhdanProperty1981}.

\begin{example}\label{eg:silly.example}
  As already remarked in \cite{intro} (see also \cite{Wang07}), an obvious example of sequences of asymptotic expander graphs that are not expanders is as follows: let $\braces{Y_n}_{n\in\NN}$ be a sequence of expanders and let $\braces{Z_n}_{n\in\NN}$ be any sequence of graphs such that $\frac{\abs{Z_n}}{\abs{Y_n}}\to 0$. Let $X_n$ be a connected graph obtained by joining arbitrarily $Y_n$ and $Z_n$ with some edges. Then $\braces{X_n}_{n\in\NN}$ is a sequence of asymptotic expanders. Furthermore, if $\abs{Z_n}\to\infty$ and the graphs $Y_n$ and $Z_n$ are connected via a single edge, then $\braces{X_n}_{n\in\NN}$ is \emph{not} a sequence of expanders. In \cite{Wang07}, sequences of graphs obtained as above are called \emph{perturbed expanders}.
\end{example}

\begin{rmk}\label{rmk:non.ce.asymptotic.expander}
 Since it is known that there exist uncountably many families of expanders that are pairwise not coarsely equivalent \cite{fisher_rigidity_2019}, it is then very easy to show that the above construction can be used to produce uncountably many non\=/coarsely\=/equivalent classes of asymptotic expanders that are not expanders.

 For example, assume that we are given for every $i\in I$ two families of expanders $\braces{Y_n^{(i)}}_{n\in\NN}$ and $\braces{Z_n^{(i)}}_{n\in\NN}$, and assume that $X_n^{(i)}$ is obtained by joining $Y_n^{(i)}$ to $Z_n^{(i)}$ with a single edge. Let $X^{(i)}$ be a coarse disjoint union. Up to reindexing over $n$, a coarse equivalence $f\colon X^{(i)}\to X^{(j)}$ must come from uniform coarse equivalences $f_n\colon X_n^{(i)}\to X_n^{(j)}$ \cite{khukhro_expanders_2017}. The edge joining $Y_n^{(i)}$ to $Z_n^{(i)}$ is a bottleneck that ``coarsely disconnects'' $X_n$---it is in fact the unique bottleneck. It is then easy to deduce that the $f_n$ must restrict to uniform coarse equivalences between $Y_n^{(i)}$ and $Y_n^{(j)}$. If the $\braces{Y_n^{(i)}}_{n\in\NN}$ are assumed to be non\=/coarsely\=/equivalent then no such $f$ can exist.
\end{rmk}

\begin{rmk}\label{rmk:gross.trick}
  It may be interesting to note that it is possible to realize asymptotic expanders that are not expanders as a sequence of Schreier coset graphs. This is in fact very easily done using ``Gross Trick'' \cite{gross1977every}: one can use the Petersen 2\=/factor Theorem to deduce that every finite $2d$\=/regular graph (possibly with loops and multiple edges) is a Schreier graph.
  One can always turn a graph of degree at most $d$ into a $2d$ regular graph by doubling all the edges and adding loops at every vertex having degree less than $2d$.
\end{rmk}

\begin{example}\label{eg:Schreier.triv.intersection}
Recall that a sequence of graphs is said to have \emph{large girth} if the length of the smallest cycle in the graphs tends to infinity. One can construct large girth examples of asymptotic expanders that are not expanders using the following modification of the previous example. Let $\braces{Y_n}_{n\in\NN}$ be a large girth sequence of expanders and let $\braces{Z_n}_{n\in\NN}$ be any large girth sequence of graphs such that $\frac{\abs{Z_n}}{\abs{Y_n}}\to 0$. Let $X_n$ be the connected graph obtained by deleting the edges $(y,y')$ in $Y_n$ and $(z,z')$ in $Z_n$, and adding two new edges $(y,z)$ and $(y',z')$. It is clear that this will be a large girth sequence of asymptotic expanders that are not expanders. Taking the $Y_n$ and $Z_n$ to be $2d$-regular graphs, the resulting $X_n$ will be Schreier coset graphs of the free group $F_d$ with respect to a sequence of finite-index subgroups with trivial intersection, thanks to the large girth condition. 

It is proved in \cite[Theorem 2.8]{Hume_continuum_2017} that there exist uncountably many non coarsely equivalent families of expanders with large girth. Making a judicious use of such families, one can choose the graphs $Z_n$ in this construction to be large girth expanders. It then follows from the argument of Remark~\ref{rmk:non.ce.asymptotic.expander} that one can construct this way uncountably many large girth asymptotic expanders that are not expanders and are not coarsely equivalent to one another.
\end{example}

\begin{rmk}
 If $\Gamma_1,\Gamma_2,\ldots $ are the finite index subgroups of $F_2$ obtained in Example~\ref{eg:Schreier.triv.intersection}, one can show that $\Gamma_n<F_2$ cannot be a \emph{normal} subgroup when $n$ is large enough (this will follow from Theorem~\ref{thm:symmetric.as.exp.iff.exp}). Furthermore, one can not expect that $\Gamma_{n+1}<\Gamma_n$, in general.

 On the other hand, Abért and Elek managed to use random methods to show (\cite[Theorem 5]{abert2012dynamical}) that there exist chains of nested subgroups $F_2>\Gamma_1>\Gamma_2>\cdots$ with trivial intersection such that the associated Schreier coset graphs are not expanders but the ``boundary action is strongly ergodic'' (see \cite{abert2012dynamical} for the relevant definitions).
 One can then show that these Schreier graphs are a sequence of asymptotic expanders \cite{li2021asymptotic}.
\end{rmk}

The following example is somewhat more technical, and it shows that asymptotic expanders need not be perturbed expanders in the sense of \cite{Wang07}. The idea is inspired from \cite{schmidtAmenabilityKazhdanProperty1981}.

\begin{example}
 Choose for every $n\in\NN$ and $k\leq n$ a positive number $0<a_{n,k}\in\NN$ such that:
 \begin{enumerate}
  \item for every fixed $k$, we have $a_{n,k}\to\infty$ as $n\to\infty$;
  \item for every fixed $n$, we have $a_{n,k}\geq a_{n,k'}$ for every $k<k'\leq n$.
  \item for every $\alpha>0$ there exists a $\bar k$ such that $\sum_{i=\bar k}^n a_{n,i}<\alpha\sum_{i=0}^n a_{n,i}$ for every $n\geq \bar k$;
  \item for every $c>0$ there exists a $\bar k$ such that $a_{n,\bar k}<c\sum_{i=\bar k}^n a_{n,i}$ for every $n$ large enough.
 \end{enumerate}
 For example, letting $a_{n,k}\coloneqq \lceil\frac{n^2}{k^2}\rceil$ satisfies all the above requirements.

 For every $k$, let $\braces{Z_{n,k}}_{n\geq k}$ be a sequence of expanders with $\abs{Z_{n,k}}=a_{n,k}$ (it is possible to produce sequences of expanders with arbitrarily prescribed cardinalities, see \emph{e.g.} \cite[Corollary 17]{de_laat_superexpanders_2019}).
 For every $n$ and $0<k\leq n$, choose injective functions $f_{n,k}\colon Z_{n,k}\to Z_{n,k-1}$. Let $X_n$ be the graph obtained by taking $\bigcup_{k\leq n} Z_{n,k}$ and adding all the edges of the form $\paren{v,f_{n,k}(v)}$ as $v$ varies in $Z_{n,k}$ and $0<k\leq n$. Finally, for every fixed $k$ let $Y_{n,k}\subset X_n$ be the full subgraph having vertices $\bigcup_{i\leq \max\braces{k,n}}Z_{n,i}$.

 For every fixed $k$, it is easy to show that the sequence $\braces{Y_{n,k}}_{n\in\NN}$ is a family of expanders. Condition (3) implies that the $Y_{n,k}$ form an exhaustion by expanders. It therefore follows from Theorem~\ref{thm:structure.of.as.exp.} that $\braces{X_n}_{n\in\NN}$ is a sequence of asymptotic expanders. On the other hand, it follows from condition (4) that for every $c>0$ there is a $k$ so that the sets $X_n\smallsetminus Y_{n,k}$ are $c$\=/\fol sets for every $n$ large enough. Hence $\braces{X_n}_{n\in\NN}$ is not a sequence of expanders.

 If the $a_{n,k}$ are chosen so that for every fixed $k$ there is a $\beta<1$ such that $\abs{Y_{n,k}}\leq \beta\abs{X_n}$ for every $n$ large enough, then the \fol sets are ``linearly big'' (this is the case for $a_{n,k}\coloneqq \lceil\frac{n^2}{k^2}\rceil$). In this case it follows that $\braces{X_n}_{n\in\NN}$ is not even a sequence of perturbed expanders.
\end{example}

\section{Ghost projections and the coarse Baum--Connes conjecture}\label{sec:ghosts}

In this section, we will use Theorem~\ref{thm:structure.of.as.exp.general.case} to show that the averaging projection $P_X$ belongs to the uniform Roe algebra $C^*_u(X)$ if and only if $X$ is a sequence of asymptotic expanders.
This will allow us to apply results from \cite{FinnSell2014} to produce uncountably many \emph{new} counterexamples to the coarse Baum--Connes conjecture in the sense that they need not have spectral gaps.

\begin{thm}\label{thm: existence of ghost proj}
Let $\{X_n\}_{n\in \N}$ be a sequence of finite metric spaces of uniformly bounded geometry with $|X_n| \to \infty$. Let $X$ be their coarse disjoint union, and $P_X$ be the averaging projection. Then the following are equivalent:
\begin{enumerate}
  \item  $P_X$ belongs to the uniform Roe algebra $C^*_u(X)$;
  \item $P_X$ is quasi-local;
  \item $\{X_n\}_{n\in \N}$ is a sequence of asymptotic expanders.
\end{enumerate}
\end{thm}

\begin{proof}
 Recall that \cite[Theorem B]{intro} states that $X$ is a coarse disjoint union of a sequence of asymptotic expanders if and only if $P_X$ is quasi\=/local. Recall also that any element of the uniform Roe algebra must be quasi\=/local (Subsection~\ref{ssec:Roe.algebras}). It is thus enough to show ``(3) $\Rightarrow$ (1)''.

 From Theorem \ref{thm:structure.of.as.exp.general.case}, there exist $c>0$, a positive decreasing sequence $\{\alpha_k\}_{k \in \N}$ with $\alpha_k \to 0$, a positive sequence $\{R_k\}_{k\in \N}$, and for each $n \in \N$ a sequence of subsets $\{Y_{n,k}\}_{k\in \N}$ in $X_n$, satisfying: for each $n,k\in \N$,
  \begin{itemize}
    \item $|Y_{n,k}| \geq (1-\alpha_k)|X_n|$;
    \item for each non-empty $A \subseteq Y_{n,k}$ with $|A| \leq \frac{1}{2}|Y_{n,k}|$, we have $|\partial_{R_k}^{Y_{n,k}} A| > c|A|$.
  \end{itemize}

Let $d$ be the metric on $X$ and fix a $k \in \N$. For each $n \in \N$, we can proceed as in the proof of Theorem~\ref{thm:intro:non.CE} and make $Y_{n,k}$ into a graph $\CG_{n,k}$ by connecting vertices $v,w\in Y_{n.k}$ with an edge if and only if $d(v,w) \leq R_k$. It follows from the assumption that $\{\CG_{n,k}\}_{n\in \N}$ is a sequence of expanders.

We denote by $d_{n,k}$ the edge-path metric induced on $Y_{n.k}$ from $\CG_{n,k}$, and we let $(Y_k,d_k)$ be a coarse disjoint union of $\{(Y_{n,k},d_{n,k})\}_{n\in \N}$. Then the averaging projection $P_{Y_k}$ belongs to the uniform Roe algebra $C^*_u(Y_k)$. Furthermore, since $d(v,w) \leq R_k \cdot d_{n,k}(v,w)$ for every $v,w \in Y_{n,k}$, we can assume that the metric $d_k$ satisfies
\begin{equation}\label{EQ1}
d(v,w) \leq R_k \cdot d_{k}(v,w)
\end{equation}
for every $v,w\in Y_k$.

Now consider the inclusion $i_k\colon Y_k \hookrightarrow X$. It induces an isometric embedding $V_k\colon \ell^2(Y_k) \to \ell^2(X)$, and we consider the adjoint homomorphism
$$\Ad_{V_k}\colon \B(\ell^2(Y_k)) \to \B(\ell^2(X)), \quad T \mapsto V_kTV_k^*$$
where $V^*_k\colon \B(\ell^2(X))\to\B(\ell^2(Y_{n,k}))$ is the adjoint of $V_k$ (which turns out to be the orthogonal projection).

From Inequality (\ref{EQ1}), we have $\Ad_{V_k}(\C_u[Y_k]) \subseteq \C_u[X]$. Since $\Ad_{V_k}$ preserves operator norms, we obtain a $C^*$-monomorphism
$$(i_k)_\ast\coloneqq \Ad_{V_k}|_{C^*_u(Y_k)}\colon C^*_u(Y_k) \to C^*_u(X).$$

For each $k$, let $Q_k\coloneqq (i_k)_\ast(P_{Y_k})=V_kP_{Y_k}V_k^*$ in $C^*_u(X)$. We can write $Q_k$ as a block\=/diagonal matrix $Q_k=(Q_{k,n})_{n\in \N}$ where $Q_{k,n} \in \B(\ell^2(Y_{n,k}))$. Direct calculation shows that for each $n$ and $x,y \in X_n$:
\begin{equation*}
(P_{X_n}-Q_{k,n})_{x,y}=
\begin{cases}
  ~-\frac{|X_n|-|Y_{n,k}|}{|Y_{n,k}| \cdot |X_n|} & \text{if } x,y \in Y_{n,k}; \\
  ~\frac{1}{|X_n|} & \mbox{otherwise}.
\end{cases}
\end{equation*}
Since each operator $P_{X_n}-Q_{k,n}$ is represented by a finite matrix, from linear algebra we know that its operator norm does not exceed its Frobenius norm:
\begin{align*}
\|P_{X_n}-Q_{k,n}\|_F^2 =& \sum_{x,y \in X_n}\big|(P_{X_n}-Q_{k,n})_{x,y}\big|^2 \\
 =& \sum_{(x,y) \in Y_{n,k}^2}\frac{(|X_n|-|Y_{n,k}|)^2}{|Y_{n,k}|^2\cdot |X_n|^2} + \sum_{(x,y) \in X_n^2\setminus Y_{n,k}^2}\frac{1}{|X_n|^2}\\
 =& \frac{(|X_n|-|Y_{n,k}|)^2+|X_n|^2 - |Y_{n,k}|^2}{|X_n|^2} \\
 \leq &\frac{\alpha_k^2 |X_n|^2 + |X_n|^2 - (1-\alpha_k)^2|X_n|^2}{|X_n|^2}\\
 =& 2\alpha_k.
\end{align*}

 Since $P_X$ and $Q_X$ are block\=/diagonal, this implies that
$$\|P_X - Q_k\| \leq \sqrt{2\alpha_k}.$$
By the assumption that $\alpha_k \to 0$, we have $Q_k \to P_X$ in the operator norm as $k \to \infty$. Since $Q_k \in C^*_u(X)$ for every $k$, we have proved that $P_X \in C^*_u(X)$.
\end{proof}

\begin{rmk}
 We restricted the statement of Theorem~\ref{thm: existence of ghost proj in Roe} to sequences of spaces of uniformly bounded geometry because we defined the (uniform) Roe algebra only for spaces of bounded geometry.
\end{rmk}

The following follows directly from Remarks~\ref{rmk:uniformRoe.embeds.Roe}, Remark \ref{rmk:average.non.compact.ghost}, and Theorem~\ref{thm: existence of ghost proj}:

\begin{cor}\label{cor: existence of ghost proj in Roe alg}
Let $\{X_n\}_{n\in \N}$ be a sequence of asymptotic expanders of bounded geometry, and $X$ be their coarse disjoint union. Then there exists a non-compact ghost projection in the Roe algebra $C^*(X)$.
\end{cor}

Furthermore, we can extend the above result to spaces containing weakly embedded asymptotic expanders in a ``proper'' way:
\begin{prop}\label{prop: weak embedding and ghost projection.finite to 1}
Let $Y$ be a metric space of bounded geometry. Assume there exists a sequence of asymptotic expanders $\{X_n\}_{n\in \N}$ of uniformly bounded geometry which admits a weak embedding $\{f_n\colon X_n \to Y\}_{n\in \N}$ such that for every $y\in Y$ the preimage $f_n^{-1}(y)$ is empty for all but finitely many $n\in\NN$.
Then there exists a non-compact ghost projection in the Roe algebra $C^*(Y)$.
\end{prop}

\begin{proof}
As in the proof of Theorem~\ref{thm: existence of ghost proj}, it follows from Theorem \ref{thm:structure.of.as.exp.general.case} that $X_n$ (weakly) contains an expander graph $\CG_n$. Since the restriction of $f_n$ to $\CG_n$ is again a weak containment, we can assume without loss of generality that $\braces{X_n}_{n\in\NN}$ is a sequence of expander graphs equipped with their edge path metrics. In this case, the existence of non\=/compact ghost projections in the Roe algebra of $Y$ is well\=/known to experts (for instance, the anonymous referee pointed out to us that analogous arguments appear \emph{e.g.} in \cite[Section 7]{HLS02}). 
 For completeness, we provide here a proof.

Let $X$ be a coarse disjoint union of $\{X_n\}_{n\in\NN}$, and $f\coloneqq \bigsqcup f_n\colon X \to Y$. Assume each $f_n$ is $L$-Lipschitz. By appropriately choosing the metric on $X$, we may assume that $f$ is $L$-Lipschitz as well. From the above assumption, we know that for any $y\in Y$, the set $f^{-1}(y)$ is finite. 

Denote $\H_X\coloneqq\ell^2(X)$ and $\H_Y\coloneqq\ell^2(Y; \ell^2(X)) \cong \ell^2(Y) \otimes \ell^2(X)$. We would like to construct an isometry $V\colon \H_X \to \H_Y$ covering $f\colon X \to Y$ in the sense that $\supp (V) \subseteq \{(f(x),x)\in Y\times X\mid x \in X\}$. For each $y \in Y$, consider the isometric embedding
\[
V_y\colon\chi_{f^{-1}(y)}\H_X \cong \ell^2(f^{-1}(y)) \longrightarrow \chi_{y}\H_Y \cong \C \delta_{y}\otimes \ell^2(X)
\]
defined by
\[
\delta_z \mapsto \delta_{f(z)} \otimes \delta_z = \delta_y \otimes \delta_z.
\]
Noting that $\H_X=\bigoplus_{y\in Y}\chi_{f^{-1}(y)}\H_X$ and $\H_Y=\bigoplus_{y\in Y}\chi_{y}\H_Y$, we can thus set
\[
V\coloneqq \bigoplus_{y\in Y} V_y\colon \H_X\longrightarrow \H_Y.
\]
It follows directly from the above construction that $\supp (V) \subseteq \{(f(x),x)\in Y\times X\mid x \in X\}$, hence $V$ induces a homomorphism
\[
\mathrm{Ad}_V\colon C^*_u(X) \cong C^*(X;\H_X) \longrightarrow C^*(Y;\H_Y) \cong C^*(Y), \quad T \mapsto VTV^*.
\]
Note that the image of $\mathrm{Ad}_V$ is indeed in the Roe algebra because $f$ is Lipschitz and finite-to-one.

Since $\{X_n\}_{n\in\NN}$ is a sequence of expanders, the associated averaging projection $P_X$ is a ghost projection and belongs to $C^*_u(X)$. We consider the operator $Q\coloneqq VP_XV^*$. Clearly, $Q$ is a non-compact projection and belongs to the Roe algebra $C^*(Y)$. Hence it suffices to show that $Q$ is ghost.

Direct calculation shows that for any $y,z \in Y$, we have
\[
Q_{y,z}=\bigoplus_{n} Q^{(n)}_{y,z} \in \prod_{n} \B(\ell^2(X_n)) \subseteq \B(\ell^2(X)),
\]
where $Q^{(n)}_{y,z} \in \B(\ell^2(X_n))$ is defined by:
\begin{equation*}
(Q^{(n)}_{y,z})_{x,x'}\coloneqq
\begin{cases}
  ~\frac{1}{|X_n|} & \text{if }x,x'\in X_n, \mbox{~and~}f_n(x)=y, f_n(x')=z; \\
  ~0 & \mbox{otherwise}.
\end{cases}
\end{equation*}
In other words, we have: 
$$Q^{(n)}_{y,z}=\chi_{f_n^{-1}(y)}P_n\chi_{f_n^{-1}(z)},$$
which implies that 
\[
\|Q^{(n)}_{y,z}\| = \frac{\sqrt{|f_n^{-1}(y)|\cdot |f_n^{-1}(z)|}}{|X_n|}.
\]

Since $\{f_n\}$ is a weak embedding, for any $\varepsilon>0$, there exists some $N_0 \in \N$ such that for every $n>N_0$ and $y \in f(X_n)$, we have
$$\frac{|f_n^{-1}(y)|}{|X_n|} <\varepsilon.$$
Consider the finite subset
\[
F\coloneqq\bigcup_{n=1}^{N_0} f_n(X_n).
\]
For every pair $y,z \in Y$ with $(y,z)\notin F \times F$, at least one of $y,z$ does not belong to $F$. This implies that $Q^{(n)}_{y,z}=0$ for every $n\leq N_0$ and hence
\[
\|Q_{y,z}\| = \sup_{n>N_0} \|Q^{(n)}_{y,z}\| = \sup_{n>N_0} \frac{\sqrt{|f_n^{-1}(y)|\cdot |f_n^{-1}(z)|}}{|X_n|} < \varepsilon.
\]
This proves that $Q$ is ghost.
\end{proof}

\begin{cor}\label{cor: weak embedding and ghost projection.box space}
Let $\{Y_n\}_{n\in \N}$ be a sequence of finite metric spaces of uniformly bounded geometry, and $Y$ be their coarse disjoint union. Assume that there exists a sequence of asymptotic expanders $\{X_n\}_{n\in \N}$ of uniformly bounded geometry which admits a weak embedding into $Y$. Then there exists a non-compact ghost projection in the Roe algebra $C^*(Y)$.
\end{cor}

\begin{proof}
From Theorem \ref{thm:structure.of.as.exp.general.case}, without loss of generality, we can assume that $\{X_n\}_{n\in\NN}$ is a sequence of expander graphs with edge path metrics. 
From Proposition \ref{prop: weak embedding and ghost projection.finite to 1}, it suffices to show that for any $y\in Y$, $f_n^{-1}(y)$ is non-empty for only finitely many $n$'s. 

Assume the contrary: there exists $y_0 \in Y$ such that $f_n^{-1}(y) \neq \emptyset$ for infinitely many $n$'s. For any such $n$, since each $X_n$ is connected and $f_n$ is $L$-Lipschitz, there exists a uniform $K_0$ (independent of $n$) such that $f(X_n)\subseteq \bigsqcup_{i=1}^{K_0} Y_i= E$. Hence we obtain a subsequence $\{X_{n_j}\}_{j\in \N}$ such that $f_{n_j}(X_{n_j}) \subseteq E$, while $E$ is a finite subset. It is not hard to show that this is a contradiction to $\{f_n\}_{n\in\NN}$ being a sequence of weak embeddings.
\end{proof}

The proof of Proposition~\ref{prop: weak embedding and ghost projection.finite to 1} would have been simpler if the image of an asymptotic expander under a weak embedding was again an asymptotic expander (this is the case for images of asymptotic expanders under \emph{coarse} embeddings, see Remark~\ref{rmk:coarse.image.of.asymptotic.expander}). The following example shows that this is not the case in general:

\begin{example}
 Let $\braces{X'_n}_{n\in\NN}$ be a sequence of expanders with $\abs{X'_n}=n\log(n)$. Let $X_n$ be the graph obtained from $X'_n$ by adding to it a string of $n$ vertices $I_n=(v_1,\ldots,v_n)$ by connecting $v_n$ to an arbitrary vertex of $X'_n$ . Further, choose a partition $\CP_n$ of $X'_n$ into $n$ subsets of cardinality $\log(n)$. Clearly, $\{X_n\}_{n\in \NN}$ is a sequence of asymptotic expanders.
 
 Let $Y_n$ be the quotient graph of $X_n$ where each subset $R\in\CP_n$ is collapsed to a single vertex. Then $\abs{Y_n}=2n$ and $\braces{Y_n}_{n\in\NN}$ is \emph{not} a sequence of asymptotic expanders because the image of $(v_1,\ldots ,v_{\frac n2})$ under the projection has size equal to $\frac{\abs{Y_n}}{4}$ but has small boundary.
 
 On the other hand, if $Y$ is a coarse disjoint union of the $Y_n$, then quotient maps $\pi_n\colon X_n\to Y_n\subset Y$ give a weak embedding. In fact, it is easy to see that for every $R\in \NN$ and $w\in Y_n$, the pre-image of the ball $B(w,R)$ under $\pi_n$ has cardinality at most $(D\log(n))^{R+1}$ where $D$ is the degree of $X_n$.
\end{example}

Using the work of Martin Finn-Sell in \cite{FinnSell2014}, it is now easy to prove the following:

\begin{thm}\label{thm: existence of ghost proj in Roe}
Let $\{X_n\}_{n\in \N}$ be a sequence of finite metric spaces of uniformly bounded geometry and $X$ be their coarse disjoint union. If $X$ admits a fibered coarse embedding into a Hilbert space, and there exists a sequence of asymptotic expanders of uniformly bounded geometry which admits a weak embedding into $X$, then the following statements hold:
\begin{itemize}
\item[(1)] The coarse Baum--Connes assembly map $\mu$ for $X$ is injective but non-surjective.
\item[(2)] The induced map $\iota_*\colon K_*(\mathcal{K})\rightarrow K_*(I_G)$ is injective but non-surjective, where $\iota\colon\mathcal{K}\hookrightarrow I_G$ is the inclusion of the compact ideal $\mathcal{K}$ into the ghost ideal $I_G$ of the Roe algebra $C^*(X)$.
\item[(3)] The induced map $\pi_*\colon K_*(C^*_{max}(X))\rightarrow K_*(C^*(X))$ is injective but non-surjective, where $\pi\colon C^*_{max}(X)\twoheadrightarrow C^*(X)$ is the canonical surjection from the maximal Roe algebra onto the Roe algebra. In particular, $\pi$ is not injective.
\end{itemize}
\end{thm}

\begin{proof}
It follows from Corollary~\ref{cor: weak embedding and ghost projection.box space} that the Roe algebra $C^*(X)$ contains a non-compact ghost projection.

(1): The injectivity of $\mu$ follows from \cite[Theorem~34]{FinnSell2014} and the non-surjectivity of $\mu$ follows from \cite[Proposition~35]{FinnSell2014}.

(2): The injectivity of $\iota_*$ follows from \cite[Lemma~33 (2)]{FinnSell2014} and its non-surjectivity follows from \cite[Proposition~35]{FinnSell2014} and part (1) of this Theorem.

(3): Let $\mu_{max}$ be the maximal coarse Baum--Connes assembly map for $X$. In particular, we have that $\pi_*\circ \mu_{max}=\mu$. From \cite[Theorem~1.1]{MR3116568} or \cite[Theorem~30]{FinnSell2014} we conclude that $\mu_{max}$ is an isomorphism. Since $\mu$ is injective, $\pi_*$ is injective as well. The non-surjectivity of $\pi_*$ follows easily from the non-surjectivity of $\mu$.
\end{proof}

One can show that if $\braces{X_n}_{n\in\NN}$ is a sequence of graphs with large girth, then $X$ admits a fibered coarse embedding into a Hilbert space \cite[Example~2.5]{MR3116568}.
Then the following is a consequence of Example~\ref{eg:Schreier.triv.intersection} and Theorem~\ref{thm: existence of ghost proj in Roe}.

\begin{cor}\label{cor:uncountably.many.counterexamples}
There are uncountably many non-coarsely-equivalent families of asymptotic expander graphs that are \emph{not} expanders and violate the coarse Baum--Connes conjecture.
\end{cor}

This also answers Question 7.2 in \cite{intro}.

\begin{rmk}\label{rmk:counterexamples.no.gap}
All the previously known counterexamples to the coarse Baum--Connes conjecture are based on the existence of a non\=/compact ghost projection $P$ in the Roe algebra $C^*(X)$. The usual strategy to construct such an operator is to consider the projection $P$ onto the $0$\=/eigenspace of some operator $\Delta\in C^*(X)$ that has a spectral gap at 0. For example, this is done when $X$ is a coarse disjoint union of a sequence of finite graphs of uniformly bounded degrees. In this case, the graphs are expanders if and only if the Laplacian operator $\Delta_X\in C^*(X)$ has a spectral gap at 0. When this happens, the averaging projection $P_X$ onto the $0$\=/eigenspace of the Laplacian belongs to $C^*(X)$ (see e.g. \cite[Examples~5.3]{MR2871145}). Something similar also happens when considering group actions with a spectral gap and the associated warped cones (see \cite{MR4000570, sawicki_warped_2017}). 

In contrast, if $X$ is a coarse disjoint union of a sequence of asymptotic expanders that are \emph{not} expanders then the Laplacian operator $\Delta_X$ does not have a spectral gap. In order to show that the averaging projection $P_X$ belongs to $C^*(X)$, we had to show explicitly that $P_X$ is a limit of operators in $C^*(X)$. Therefore, Corollary~\ref{cor:uncountably.many.counterexamples} provides a new class of counterexamples to the coarse Baum--Connes conjecture that are not directly derived from actions with a spectral gap.
\end{rmk}

We end this section by pointing out that we can use the works of Finn-Sell \cite{MR3500819} and Higson--Lafforgue–-Skandalis \cite{HLS02} to prove that asymptotic expanders also provide counterexamples to the group Baum--Connes conjecture with coefficients:

\begin{cor}
Let $\Gamma$ be a finitely generated discrete group, and $\braces{X_n}_{n\in\NN}$ be a sequence of asymptotic expander graphs with uniformly bounded degree.
\begin{enumerate}
  \item If $\braces{X_n}_{n\in\NN}$ weakly embeds into $\Gamma$, then there exists a separable $\Gamma$-$C^*$-algebra $A$ such that the Baum--Connes assembly map for $\Gamma$ with coefficients in $A$ fails.
  \item If $\braces{X_n}_{n\in\NN}$ has large girth and coarsely embeds into $\Gamma$, then there exists a separable $\Gamma$-$C^*$-algebra $A$ such that the Baum--Connes assembly map for $\Gamma$ with coefficients in $A$ is injective but fails to be surjective.
\end{enumerate}
\end{cor}

\begin{proof}
For the first case, it follows from Theorem~\ref{thm:structure.of.as.exp.} that $\Gamma$ weakly contains a sequence of expanders. Then the proof follows from \cite[Section 7]{HLS02}.

For the second case, we simply use \cite[Corollary 3.6]{MR3500819} and Theorem \ref{thm: existence of ghost proj in Roe} (note that we have generalised \cite[Theorem 4.4]{MR3500819}).
\end{proof}

\section{Asymptotic expanders in the case of homogeneity}\label{sec:homog}

The aim of this section is to show that if all subsets realising the Cheeger constant in a graph are ``small'', then the graph must be highly nonhomogeneous. In the case of Cayley graphs, similar observations were made by Marc Lackenby (Lemma 2.1 of \cite{Lac}).
We prove the following:

\begin{lem}\label{lem:big.Cheeger}
Let $X$ be a finite graph. Then one of the following must hold:
\begin{itemize}
\item
there exists a subset $A\subset X$ with $|X|/4 < |A| \leq |X|/2$ realising the Cheeger constant of $X$;
\item
there exists a unique maximal (with respect to inclusion) subset $A\subset X$ with $|A| \leq |X|/2$ realising the Cheeger constant of $X$.
\end{itemize}
\end{lem}

\begin{proof}

Let $A\subset X$ be a subset satisfying $|A| \leq |X|/2$ realising $h(X)$ that is maximal with respect to inclusion.
Suppose that there exists another subset $B\neq A$ with $|B| \leq |X|/2$, realising the Cheeger constant (if there is no such $B$, then we are done).

If $|A\cup B|> |X|/2$, then one of $|A|$ or $|B|$ must be $>|X|/4$, and so we are done.

Suppose that $|A\cup B|\leq |X|/2$.
Let us first remark the following:
\begin{align*}
|\partial(A\cup B)| &= |\partial A| + |\partial B| - |\partial A \cap \partial B| - |A \cap \partial B| - |\partial A \cap B|;\\
|\partial (A \cap B)|&\leq |A \cap \partial B| + |\partial A \cap B| + |\partial A \cap \partial B|.
\end{align*}

Combining these, we have
\begin{align*}
|\partial (A\cup B)| &\leq |\partial A| + |\partial B| - |\partial (A\cap B)|\\
&\leq h(X)\left( |A| + |B| - |A\cap B|\right) \\
&= h(X) |A\cup B|.
\end{align*}

We must have equality throughout, by the definition of the Cheeger constant.
By the maximality of $A$, we must have $A\cup B=A$. Thus $B\subset A$ and so $A$ is the unique maximal subset with $|A| \leq |X|/2$ realising the Cheeger constant of $X$.
\end{proof}

We can thus deduce that a very weak symmetry condition---that of having more than one maximal subset realising the Cheeger constant---is sufficient for asymptotic expanders to be equivalent to expanders:
\begin{thm}\label{thm:symmetric.as.exp.iff.exp}
Let $\braces{X_n}_{n\in\NN}$ be a sequence of finite connected graphs with uniformly bounded degree and $|X_n|\rightarrow \infty$ such that each $X_n$ admits more than one maximal (with respect to inclusion) subset realising the Cheeger constant $h(X_n)$. Let $\braces{Y_n}_{n\in\NN}$ be a sequence of finite connected graphs with uniformly bounded degree which is uniformly coarsely equivalent to $\braces{X_n}_{n\in\NN}$.
Then $\braces{Y_n}_{n\in\NN}$ is a sequence of asymptotic expanders \emph{if and only if} it is a sequence of expanders.
\end{thm}

\begin{proof}
Since both the properties of ``being a sequence of expanders'' and ``being a sequence of asymptotic expanders'' are invariant under uniform coarse equivalence (\cite[Theorem 3.11 and Proposition A.2]{intro}), we may in the sequel assume that $X_n=Y_n$ is a finite graph for every $n\in \N$.

The ``if'' implication is obvious. For the other implication, by Lemma \ref{lem:big.Cheeger}, for every $n\in \NN$, we can choose a maximal subset $A_n\subset X_n$ realising $h(X_n)$ such that $|X_n|/4 < |A_n| \leq |X_n|/2$. If $X_n$ is a sequence of asymptotic expander graphs, then there is a constant $c_{1/4}>0$ such that $X_n$ has no $c_{1/4}$\=/\fol sets of cardinality greater than $|X_n|/4$. It follows that $A_n$ cannot be a $c_{1/4}$\=/\fol set, and hence $h(X_n)>c_{1/4}$ for every $n$.
\end{proof}

\begin{rmk}\label{vertex-tran}
We can apply Theorem~\ref{thm:symmetric.as.exp.iff.exp} when $\braces{X_n}_{n\in\NN}$ is a sequence of vertex-transitive graphs\footnote{Recall that a graph is vertex-transitive if for every pair of vertices $x,y\in X$ there is a graph automorphism $\phi\in\aut(X)$ such that $\phi(x)=y$.}, such as Cayley graphs. Indeed, it is clear that such graphs always admit more than one maximal subset realising the Cheeger constant (by simply taking images of such a subset under graph automorphisms).

In particular, this immediately applies in the case of a box space with respect to a nested sequence of \emph{normal} finite-index subgroups. That is, if $\Gamma=\angles{S}$ is a finitely generated infinite group and $N_1, N_2, \dots$ is a sequence of finite index normal subgroups of $\Gamma$ such that $\bigcap_{n\in \N} N_n=\{e\}$, then the \emph{box space} of $\Gamma$ with respect to $(N_n)$ is the sequence of Cayley graphs $X_n\coloneqq\cay(\Gamma/N_n,\bar{S})$ (where $\bar S$ is the image of $S$ under the quotient map). It is worth to noticing that the normality of the subgroups is essential for this conclusion to hold, as evidenced by Example \ref{eg:Schreier.triv.intersection}.
\end{rmk}

It may be interesting to note that, under some assumption of homogeneity, we can now produce a $C^*$\=/algebraic characterisation of expanders:

\begin{cor}\label{cor:symmetric.Roe}
Let $\braces{X_n}_{n\in\NN}$ be a sequence of finite connected graphs with uniformly bounded degree and $|X_n|\rightarrow \infty$ such that each graph $X_n$ admits more than one maximal (with respect to inclusion) subset realising the Cheeger constant $h(X_n)$. If $X$ is their coarse disjoint union, then the following are equivalent:
\begin{enumerate}
 \item[(1)] $\{X_n\}_{n\in\NN}$ is a sequence of expanders;
 \item[(2)] the averaging projection $P_X$ belongs to the uniform Roe algebra $C^*_u(X)$.
\end{enumerate}
\end{cor}

\section{Questions and comments}\label{sec:questions}
As already mentioned, we feel that the notion of asymptotic expanders will end up playing some role in a rather diverse array of topics. This is mostly due to the fact that this notion seems to encapsulate many good properties of genuine expanders into a robust framework.
One advantage that expanders have over asymptotic expanders is that they already have many different characterisations. One characterisation that has proved to be particularly useful is given in terms of spectral gap of the averaging operator.
This raises the following:

\begin{qu}\label{qu:spectral.char.as.exp}
In analogy to the characterisation of expanders in terms of the spectral gap of the Laplacian, is there a characterisation of asymptotic expanders in terms of spectra?
\end{qu}

A positive answer to the above question could also be useful in the setting of dynamical systems, as asymptotic expansion is equivalent to strong ergodicity \cite{li2021asymptotic}.

Going to the opposite direction of Question~\ref{qu:spectral.char.as.exp} we have:

\begin{qu}
 Is there a $C^*$\=/algebraic characterisation of expanders?
\end{qu}

One partial result in this direction follows from Corollary~\ref{cor:symmetric.Roe}, and shows that under the additional homogeneity assumption expanders can be characterised by the property that $P_X$ belongs to the uniform Roe algebra.
In general, if $\braces{X_n}_{n\in\NN}$ is a sequence of uniformly bounded degree graphs and $P_X$ is in the uniform Roe algebra then we know from Theorem~\ref{thm: existence of ghost proj} that  all the ``linearly large'' subsets of $X_n$ must have large boundary. It would then remain to find a $C^*$\=/algebraic condition that imply that ``linearly small'' subsets have large boundary.

In view of Theorem~\ref{thm:symmetric.as.exp.iff.exp}, it would be interesting to know for which classes of graphs the notions of asymptotic expanders and usual expanders are equivalent.

\begin{qu}
Which properties would guarantee that for a sequence of graphs, being asymptotic expanders is equivalent to being expanders? Is it possible to give a complete characterisation of such graphs?
\end{qu}

 To our knowledge, all the proofs that expanders do not coarsely embed into a Banach space $E$ make use of Poincaré inequalities. It follows from the discussion in Section~\ref{sec:non.coarse.embeddability} that these proofs also imply that bounded geometry asymptotic expanders cannot coarsely embed into such $E$.
 On the other hand, it is conceivable that there exists a Banach space $E$ where expanders cannot coarsely embed even if they do not satisfy an $E$\=/valued Poincaré inequality. In this case, the arguments in Section~\ref{sec:non.coarse.embeddability} would not prove that asymptotic expanders cannot coarsely embed in $E$.

 It may thus be of interest to investigate the following.
\begin{qu}
Do the following two classes of Banach spaces coincide?
 \begin{align*}
  \CC_1&\coloneqq\braces{E\mid\text{no sequence of expanders can coarsely embed into E}},\\
  \CC_2&\coloneqq\braces{E\mid\text{no bound. geom. sequence of asymptotic expanders can coarsely embed into }E}.
 \end{align*} 
(Note that $\CC_2\subseteq \CC_1$ because expanders are asymptotic expanders of bounded geometry). 
\end{qu}

\bibliographystyle{plain}
\bibliography{ExpanderishVig,bibfile}

\end{document}